\documentclass{article}
\usepackage{comment}
\usepackage[english]{babel}

\usepackage[letterpaper,top=2cm,bottom=2cm,left=3cm,right=3cm,marginparwidth=1.75cm]{geometry}

\usepackage{amssymb, amsmath, amsthm, mathtools}
\usepackage{amsfonts,setspace}
\usepackage{dsfont}
\usepackage{graphicx}
\usepackage[colorlinks=true, allcolors=blue]{hyperref}
\usepackage{enumerate}
\usepackage{mdframed}
\usepackage{tikz}
\usetikzlibrary{arrows.meta}
 
\theoremstyle{plain}
    \newtheorem{definicion}{Definition}
    \newtheorem{teorema}{Theorem}
    \newtheorem{corolario}{Corollary}[teorema]
    \newtheorem{lema}{Lemma}
    \newtheorem{proposicion}{Proposition}

    \newtheorem{ejemplo}{Example}
    \newtheorem{remark}{Remark}

\title{Maximal Monotone Differential Inclusions with Volterra and One-Sided Lipschitz Perturbations under Nonlocal Initial Conditions and Applications}

\author{
    Abderrahim Jourani$^{1}$ and Manuel Torres-Valdebenito$^{1,2}$\\
    $^{1}$Universit\'e Bourgogne Europe, CNRS, IMB UMR 5584, 21000 Dijon, France\\
    $^{2}$Departamento de Ingenier\'ia Matem\'atica, Universidad de Chile, Santiago, Chile\\
    \texttt{abderrahim.jourani@ube.fr}, \texttt{manuel-alejandro.torres-valdebenito@ube.fr}
}

\begin{document}
\maketitle

\begin{abstract}
    We investigate a class of differential inclusions governed by non-autonomous and autonomous maximal monotone operators, involving set-valued perturbations with a Volterra integral term and subject to a nonlocal condition. Under suitable assumptions on the governing operator, the set-valued perturbation, and the Volterra kernel, we establish existence results for solutions. In particular, the set-valued perturbation is assumed to satisfy a one-sided Lipschitz condition, while the Volterra kernel is required to be Lipschitz continuous. The existence of a trajectory is established by means of an iterative construction and an application of Zorn's lemma. Finally, several examples are presented to illustrate the applicability of the abstract results.
    \vskip 0.2cm
    \textbf{Keywords:} Differential inclusion, Monotone operator, Voltera kernel, One-sided Lipschitz, Gr\"onwall inequality, Fixed point theorem. \\
    
    \textbf{2000 Mathematics Subject Classification:} 49J40, 47J20, 34G25, 35K90, 47J35
\end{abstract}

\section{Introduction}

The theory of differential inclusions has become one of the fundamental branches of modern nonlinear analysis, extending the classical theory of ordinary differential equations by allowing the derivative of the unknown function to belong to a set-valued map rather than being determined by a single-valued vector field. This generalization provides a natural mathematical framework for describing dynamical systems with discontinuities, constraints, hysteresis, uncertainty, memory effects, and nonsmooth phenomena.

The origins of differential inclusions can be traced back to the pioneering work of Zaremba \cite{Zaremba1936}
and H. Marchaud \cite{MR1505014}
in the 1930s, who considered differential equations with discontinuous right-hand sides. A systematic theory, however, emerged during the 1960s through the works referenced in the books of Aubin-Cellina \cite{MR755330, MR494247} and Filippov \cite{MR1028776} and in the papers \cite{MR1042661} and \cite{MR2519593}. In particular, Filippov developed a rigorous framework for differential equations with discontinuous vector fields by replacing discontinuous functions with suitable set-valued maps, thereby laying the foundations of the modern theory of differential inclusions. Since then, the subject has experienced remarkable development, driven both by theoretical advances and by numerous applications in control theory, mechanics, optimization, economics, biology, and engineering.

Among the many classes of differential inclusions, evolution inclusions governed by maximal monotone operators have attracted considerable attention because of their strong analytical properties and their close relationship with convex analysis. The modern theory of maximal monotone operators originated in the seminal works of Minty \cite{MR169064} and Rockafellar \cite{MR261415} during the early 1960s and was further developed by Brézis, whose celebrated monograph \cite{MR348562} established the foundations of nonlinear evolution equations associated with monotone operators in Hilbert spaces.

Let $\mathcal{H}$ be a real Hilbert space and let $A\colon\mathcal{H}\rightrightarrows \mathcal{H}$ be a maximal monotone operator. A typical evolution inclusion is written as:
\begin{align*}
    \begin{cases}
        0\in \dot{x}(t)+A(x(t)),\\
        x(0)=x_0
    \end{cases}
\end{align*}
where the multivalued operator $A$ models nonlinear dissipation, constraints, or nonsmooth constitutive laws. Typical examples include subdifferentials of proper lower semicontinuous convex functions, normal cone operators associated with closed convex sets, and nonlinear elliptic differential operators. The maximal monotonicity of $A$ guarantees the well-posedness of the associated Cauchy problem under appropriate assumptions and provides powerful approximation tools through the resolvent and Yosida regularization.

In many realistic models, however, the evolution of the system is influenced by uncertainties, external disturbances, feedback mechanisms, or nonunique responses that cannot be represented by single-valued perturbations. This motivates the study of differential inclusions with set-valued perturbations of the form
\begin{align*}
    \begin{cases}
        0\in \dot{x}(t)+A(x(t)) - F(t,x(t)),\\
        x(0)=x_0
    \end{cases}
\end{align*}
where $F\colon[0,T]\times\mathcal{H}\rightrightarrows\mathcal{H}$ is a set-valued mapping with nonempty values. Such perturbations naturally arise in control systems with uncertain inputs, differential games, viability theory, optimization, and systems subject to bounded disturbances. Depending on the assumptions imposed on $F$, such as upper semicontinuity, measurable dependence, compactness, convexity of values, or suitable growth conditions, one can establish existence results by combining measurable selection theorems, compactness arguments, fixed-point techniques, and monotone operator methods.

Many physical and biological processes also exhibit hereditary or memory effects, meaning that the present state depends not only on the current configuration but also on the entire past history of the system. Such phenomena are naturally modeled by Volterra integral operators. Consequently, one is led to study evolution inclusions containing memory terms of the form
\begin{align*}
    \begin{cases}
        0 \in \dot{x}(t) + A(x(t)) - F(t,x(t))-  \int_{0}^{t} g(t,s,x(s))ds,\\
        x(0)=x_0
    \end{cases}
\end{align*}
where the integral operator
\begin{align*}
    V(x)(t)=\int_{0}^{t}g(t,s,x(s))ds
\end{align*}
represents the accumulated influence of previous states. Here the mapping  $g\colon[0,T]\times[0,T]\times\mathcal{H}\mapsto \mathcal{H}$ is a Volterra kernel. Volterra-type terms appear naturally in viscoelasticity, heat conduction with memory, population dynamics, epidemiology, neural networks, and many models of materials possessing hereditary characteristics. The presence of such memory operators considerably enriches the mathematical model while introducing significant analytical challenges due to the nonlocal dependence on the solution history. For example, this abstract differential inclusion provides a unified framework for many biological systems. Specifically,
\begin{itemize}
    \item maximal monotone operators describe biological constraints such as positivity, carrying capacities, or limited resources;
    \item set-valued perturbations represent uncertainty in growth rates, transmission coefficients, or control strategies;
    \item Volterra integral terms naturally model memory effects including incubation periods, waning immunity, maturation delays, environmental contamination, and cumulative resource consumption.
\end{itemize}

The simultaneous presence of a maximal monotone operator, a set-valued perturbation, and a Volterra integral term leads to a broad class of nonlinear evolution inclusions capable of modeling constrained dynamical systems with uncertainty and memory. Their mathematical analysis combines techniques from monotone operator theory, convex analysis, measurable set-valued mappings, nonlinear functional analysis, and fixed-point theory. The Yosida approximation of maximal monotone operators, together with compactness methods and measurable selection techniques, provides an effective framework for establishing the existence, regularity and qualitative properties of solutions. 

In many applications, the behavior of a system cannot be adequately described by a classical initial condition such as ($x(0)=x_0$). Instead, the initial or boundary condition may depend on the values of the unknown function at several points or on its values throughout the interval of consideration. Such conditions are referred to as nonlocal conditions. Typical examples include multipoint conditions $x(0)=\sum_{i=1}^{m}\alpha_i x(t_i),$
and integral conditions $x(0)=\int_0^T g(t,x(t)),dt.$

Nonlocal conditions arise naturally in various applications where the state of a system at one time is influenced by its past, future, or overall behavior.
 The combination of differential inclusions and nonlocal conditions leads to a class of problems of the form
\begin{align*}
    \begin{cases}
        0 \in \dot{x}(t) + A(t)(x(t)) - F(t,x(t))-  \int_{0}^{t} g(t,s,x(s))ds,\\
        x(0)=h(x(\cdot))    
    \end{cases}
\end{align*}
where $h\colon C([0,T], \mathcal{H}) \longrightarrow \overline{D(A(0))}$ is a nonlocal operator and for each $t\in [0, T]$,  $A(t) : D(A(t))\subset\mathcal{H} \rightrightarrows \mathcal{H}$ is a maximal monotone operator. Such problems are more general than classical initial-value problems and present additional mathematical difficulties. 

The study of nonlocal differential inclusions has attracted considerable attention because of its theoretical importance and its applications to control theory, optimization, mechanics, and other areas of applied mathematics (see \cite{MR3908332}, \cite{MR2820290},  \cite{MR3241568}, \cite{MR1910880}, \cite{ZHU2012523} and references therein). In \cite{MR1910880}, the authors established existence results under the One Sided Lipschitzness of the perturbation $F$ and the maximal monotonicity of the autonomous maximal monotone operator $A$ and without Voltera term which complicates the study. The OSL property is crucial in the existence of solution of these kind of problems. Indeed, the (upper semi-) continuity alone is not sufficient, as it is shown in the counter-example in  \cite{MR3682176}, where $F$ is a single-valued continuous (bounded)  mapping, $A=N(C, \cdot)$, the normal cone in the sense of convex analysis to a closed convex (bounded)  set $C$,  and $g\equiv 0$. 

The purpose of this work is to investigate differential inclusions governed by non-autonomous maximal monotone operators with set-valued perturbations involving a Volterra integral term under a nonlocal condition. The main objective is to establish existence results under suitable assumptions on the governing operator, the set-valued mapping, and the Volterra kernel, namely the Vladimirov absolute continuity of the set-valued operator $A$,  the One Sided Lipschitzness of the set-valued perturbation $F$ and the Lipschitz continuity of the Volterra kernel $g$.  Our approach is based on new iterative schemes, allowing us at the end to apply Zorn lemma to produce a maximal trajectory. 
The analysis developed herein contributes to the theory of nonlinear evolution inclusions and provides a mathematical framework applicable to a wide variety of models arising in mechanics, control theory, optimization, and systems with hereditary effects. 

Our approach enables us to address not only the non-autonomous case but also the autonomous one under weak assumptions. 
In particular, it allows us to consider sweeping processes involving a Volterra term, both in the convex and, more generally, in the prox-regular setting. Such problems have already been investigated in the literature under the assumption that the perturbation $F$ is Lipschitz continuous (see, for instance, \cite{MR4614261} and \cite{MR4894870}, and the references therein). In a recent work \cite{jouraninarvaezvilches2026}, the authors consider Volterra sweeping processes where the set-valued perturbations satisfies a compactness condition and the maximal monotone term is replaced by the normal cone of prox-regular sets. Note that in the present work neither the measure of noncompactness in $F$ nor the compactness of the semigroup generated by $A$ is required. The results obtained are illustrated by some examples.

\paragraph{Organization of the paper:} Section 2 recalls some preliminaries on maximal monotone operators, the measurability of set-valued mappings, and several tools from nonlinear analysis, including Grönwall-type lemmas. Section 3 presents the main assumptions and an existence result for non-autonomous differential inclusions governed by maximal monotone operators.

In Section 4, these results are applied to establish the main existence theorem for solutions and $\varepsilon$-solutions of differential inclusions involving maximal monotone operators and set-valued perturbations with a Volterra integral term under a nonlocal condition. Section 5 provides examples and remarks illustrating the verification of the main hypotheses and the applicability of the obtained results. Section 6 concerns some conclusions. Finally, the Appendix contains the complete proofs of the results stated in Section 3.

\section{Preliminaries}

Let $\mathcal{H}$ be a {\it separable Hilbert space}, endowed with the inner product $\langle\cdot,\cdot\rangle$ and the induced norm $\|\cdot\|$. Given $x\in\mathcal{H}$ and $r>0$, the closed ball centered at $x$ with radius $r$ is denoted by $\mathbb{B}(x;r):=\{y\in\mathcal{H}:\|x-y\|\leq r\}$ and, for simplicity, the notation $r\mathbb{B}:=\mathbb{B}(0;r)$ is adopted.  The interval $[0,T]$ is endowed with the $\sigma$-field $\mathcal{T}$ of its Lebesgue measurable subsets, which is complete with respect to the ($\sigma$-finite) Lebesgue measure, and $\mathcal{B}$ denotes the Borel $\sigma$-field of $\mathcal{H}$. 

\paragraph{Maximal Monotone Operators and Semigroups:} Let $A\colon\mathcal{H}\rightrightarrows\mathcal{H}$ be a multi-valued operator. The domain and the graph of $A$ are defined, respectively, by:
\begin{align*}
    D(A) &:= \{x\in\mathcal{H}:A(x)\neq\emptyset\},\\
    G(A) &:= \{(x,y)\in\mathcal{H}\times\mathcal{H} : y\in A(x)\}.
\end{align*}
The inverse operator $A^{-1}\colon\mathcal{H}\rightrightarrows\mathcal{H}$ of $A$ is defined by:
\begin{eqnarray*}
    A^{-1}(y) := \{x\in\mathcal{H} : y\in A(x)\}.
\end{eqnarray*}
Let $A\colon\mathcal{H}\rightrightarrows\mathcal{H}$ be a multivalued operator.  The operator $A$ is called monotone if for any $(x_{1},y_{1}),(x_{2},y_{2})\in\operatorname{gr}(A)$, one has $\langle y_{1}-y_{2},x_{1}-x_{2}\rangle\geq0$. The monotone operator $A$ is called maximal monotone, if there is no monotone operator $B$ such that $\operatorname{gr}(A)$ is contained strictly in $\operatorname{gr}(B)$.
\paragraph{Elements About Measurability of Multivalued Mappings:} A mapping $f\colon[0,T]\to\mathcal{H}$ is said to be strongly measurable if it is measurable with respect to the Borel $\sigma$-algebra generated by the norm topology of $\mathcal{H}$. In contrast, $f$ is said to be weakly measurable if it is measurable with respect to the Borel $\sigma$-algebra generated by the weak topology of $\mathcal{H}$, and both notions coincide when $\mathcal{H}=\mathbb{R}^{n}$. Let $F\colon[0,T]\rightrightarrows\mathcal{H}$ be a multi-valued mapping. It is said to be strongly measurable (or simply \textit{measurable}) if, for every open set $V\subseteq\mathcal{H}$, the set $\{t\in[0,T]:F(t)\cap V\neq\emptyset\}$ is measurable. Moreover, $F$ is said to admit a strongly measurable selection if there exists a strongly measurable mapping $f\colon[0,T]\to\mathcal{H}$ such that $f(t)\in F(t)$ for a.e. $t\in[0,T]$.

A multivalued mapping $G\colon\mathcal{H}\rightrightarrows\mathcal{H}$ is called upper hemicontinuous (resp. lower  hemicontinuous) if for every $y\in\mathcal{H}$, the support function $x\in\mathcal{H}\mapsto\sigma_{G(x)}(y):=\sup\{\langle y,v\rangle :v\in G(x)\}\in\overline{\mathbb{R}}$ is upper semicontinuous (resp. lower semicontinuous). $G$ is hemicontinuous if it is upper and lower hemicontinuous. It is not difficult to see that if the multivalued mapping $G\colon\mathcal{H}\rightrightarrows\mathcal{H}$ is weakly compact-valued and upper semicontinuous with respect to the strong$\times$weak topology, then $G$ is upper hemicontinuous. 


\paragraph{Elements About Nonlinear Analysis:} The following version of the differential Gr\"onwall inequality is recalled, as it will be used to derive a priori estimates for the trajectories under consideration.

\begin{lema}[Differential Gr\"onwall inequality]
    \label{lem:DifferentialGronwallInequality}
    Let $a,b\colon[0,T]\to\mathbb{R}_{+}$ be two Lebesgue integrable functions and let $\theta\colon[0,T]\to\mathbb{R}_{+}$ be an absolutely continuous function, such that:
    \begin{eqnarray*}
        \dot{\theta}(t) \leq a(t) + b(t)\theta(t) \quad\textrm{a.e. }t\in[0,T].
    \end{eqnarray*}
    Then, for every $0\leq s<t\leq T$, it holds that:
    \begin{align*}
        \theta(t) \leq \theta(0)\exp\left[\int_{0}^{t} b(s)ds\right] + \int_{0}^{t}a(s) \exp\left[ \int_{s}^{t} b(\tau)d\tau \right]ds \quad\forall t\in[0,T],
    \end{align*}
    Moreover, if $a(\cdot)$ is increasing, then:
    \begin{align*}
        \theta(t) \leq \theta(0)\exp\left[\int_{0}^{t} b(s)ds\right] + a(t)\int_{0}^{t} \exp\left[ \int_{s}^{t} b(\tau)d\tau \right]ds \quad\forall t\in[0,T].
    \end{align*}
\end{lema}

The previous result applies to standard differential inequalities. However, the framework under consideration involves more intricate estimates due to the presence of memory terms. The following enhanced version of Gr\"onwall's inequality \cite{MR4836331} will be instrumental in deriving the required bounds.

\begin{lema}[[Enhanced differential Gr\"onwall inequality]]
    \label{lem:EnhancedDifferentialGronwallInequality}
    Let $k_{1},k_{2},k_{3},k_{5}\colon[0,T]\to\mathbb{R}_{+}$ and $k_{4},k_{6}\colon [0,T]^{2}\to\mathbb{R}_{+}$ be integrable functions, where $k_{4}\in L^{1}(D)$ and $k_{6}\in L^{2}(D)$ with $D:=\{(t,s)\in[0,T]^{2}:s\leq t\}$, and let $\theta\colon[0,T]\to\mathbb{R}_{+}$ be an absolutely continuous function, such that:
    \begin{align*}
        \dot{\theta}(t) \leq
        &\varepsilon(t) + k_{1}(t)\sqrt{\theta(t)} + k_{2}(t)\theta(t) \\
        &+ k_{3}(t)\sqrt{\theta(t)} \int_{0}^{t} k_{4},(t,s)\sqrt{\theta(s)}ds + k_{5}(t)\int_{0}^{t} k_{6}(t,s)\sqrt{\theta(s)}ds  \quad\textrm{a.e. }t\in[0,T].
    \end{align*}
    Then:
    \begin{align*}
        \theta(t) \leq
        & \theta(0) \exp\left[ \int_{0}^{t} \left(k_{1}(s) + k_{2}(s) + k_{3}(s) \int_{0}^{s}k_{4}(s,\tau)d\tau + \frac{k_{5}(s)s}{2}\right) ds\right]\\
        & + \int_{0}^{t} \left(\varepsilon(s) + k_{1}(s)+ \frac{k_{5}(s)}{2}\int_{0}^{s}k_{6}^{2}(s,\tau)d\tau\right) E(t,s)ds\qquad\forall t\in[0,T],
    \end{align*}
    where
    \begin{eqnarray*}
        E(t,s):=\exp\left[\int_{s}^{t} \left(k_{1}(\tau) + k_{2}(\tau) + k_{3}(\tau) \int_{0}^{\tau}k_{4}(\tau,\sigma)d\sigma + \frac{k_{5}(\tau)\tau}{2} \right)d\tau \right].
    \end{eqnarray*}
\end{lema}

\begin{proof}
    Let $\vartheta\colon[0,T]\to\mathbb{R}_{+}$ be the absolutely continuous function satisfying $\vartheta(0)=\theta(0)$ and
    \begin{align*}
        \dot{\vartheta}(t):=
        &\varepsilon(t) + k_{1}(t)\sqrt{\theta(t)} + k_{2}(t)\theta(t) \\
        &+ k_{3}(t)\sqrt{\theta(t)} \int_{0}^{t} k_{4}(t,s)\sqrt{\theta(s)}ds + k_{5}(t)\int_{0}^{t} k_{6}(t,s)\sqrt{\theta(s)}ds  \quad\textrm{a.e. }t\in[0,T],
    \end{align*}
    Therefore, $\theta(t)\leq\vartheta(t)$ for every $t\in[0,T]$. Hence, for a.e. $t\in[0,T]$,
    \begin{align*}
        \dot{\vartheta}(t) 
        \leq 
        &\varepsilon(t) + k_{1}(t)\sqrt{\theta(t)} + k_{2}(t)\theta(t) \\
        &+ k_{3}(t)\sqrt{\theta(t)} \int_{0}^{t} k_{4}(t,s)\sqrt{\theta(s)}ds + k_{5}(t)\int_{0}^{t} k_{6}(t,s)\sqrt{\vartheta(s)}ds\\
        \leq 
        &\varepsilon(t) + k_{1}(t)\sqrt{\vartheta(t)}+k_{2}(t)\vartheta(t)+\left(k_{3}(t)\int_{0}^{t}k_{4}(t,s)ds\right)\vartheta(t) +\frac{k_{5}(t)}{2}\int_{0}^{t} (k_{6}^{2}(t,s) + \vartheta(s))ds \\
        \leq
        &\varepsilon(t) + \frac{k_{5}}{2}\int_{0}^{t}k_{6}^{2}(t,s)ds + k_{1}(t)\sqrt{\vartheta(t)} + \left(k_{2}(t) + k_{3}(t) \int_{0}^{t}k_{4}(t,s)ds + \frac{k_{5}(t)t}{2} \right)\vartheta(t) ,
    \end{align*}
    where the fact that $\vartheta$ is nondecreasing and Young's inequality applied to the product of the nonnegative terms $k_{6}(t,s)\sqrt{\vartheta(s)}$ have been used. Moreover, since $\sqrt{\vartheta(t)}\leq\vartheta(t)+1$, it follows that, for a.e. $t\in[0,T]$,
    \begin{eqnarray*}
        \dot{\vartheta}(t) \leq \varepsilon(t) + k_{1}(t)+ \frac{k_{5}(t)}{2}\int_{0}^{t}k_{6}^{2}(t,s)ds + \left(k_{1}(t) + k_{2}(t) + k_{3}(t) \int_{0}^{t}k_{4}(t,s)ds + \frac{k_{5}(t)t}{2} \right)\vartheta(t).
    \end{eqnarray*}
    Applying Lemma \ref{lem:DifferentialGronwallInequality}, the following estimate is obtained:
    \begin{align*}
        \vartheta(t) 
        \leq
        & \vartheta(0) \exp\left[ \int_{0}^{t} \left(k_{1}(s) + k_{2}(s) + k_{3}(s) \int_{0}^{s}k_{4}(s,\tau)d\tau + \frac{k_{5}(s)s}{2}\right) ds\right]\\
        & + \int_{0}^{t} \left(\varepsilon(s) + k_{1}(s)+ \frac{k_{5}(s)}{2}\int_{0}^{s}k_{6}^{2}(s,\tau)d\tau\right) E(t,s)ds\qquad\forall t\in[0,T],
    \end{align*}
    where $E(t,s)$ is given by Lemma \ref{lem:EnhancedDifferentialGronwallInequality}. Therefore,
    \begin{align*}
        \theta(t) \leq
        & \theta(0) \exp\left[ \int_{0}^{t} \left(k_{1}(s) + k_{2}(s) + k_{3}(s) \int_{0}^{s}k_{4}(s,\tau)d\tau + \frac{k_{5}(s)s}{2}\right) ds\right]\\
        & + \int_{0}^{t} \left(\varepsilon(s) + k_{1}(s)+ \frac{k_{5}(s)}{2}\int_{0}^{s}k_{6}^{2}(s,\tau)d\tau\right) E(t,s)ds\qquad\forall t\in[0,T],
    \end{align*}
\end{proof}

\begin{lema}[Integral Gr\"onwall inequality]
    \label{lem:IntegralGronwallInequality}
    Let $\theta\colon[0,T]\to\mathbb{R}_{+}$ be a continuous function. Consider the integrable functions $a\colon[0,T]\to\mathbb{R}_{+}$ and $b\colon[0,T]\to\mathbb{R}_{+}$ such that, for every $t\in[0,T]$, the following inequality holds:
    \begin{equation*}
        \theta(t) \leq \theta(0) + a(t) + \int_{0}^{t} b(s)\theta(s)ds.
    \end{equation*}
    Then, the following estimate holds:
    \begin{equation*}
        \theta(t) \leq \theta(0) + a(t)
        + \int_{0}^{t}
        \big(\theta(0)+a(s)\big)b(s)
        \exp\left(\int_{s}^{t}b(\tau)d\tau\right)ds
        \quad\forall t\in[0,T].
    \end{equation*}
\end{lema}

\begin{proof}
    Let $\vartheta\colon[0,T]\to\mathbb{R}_{+}$ be the absolutely continuous function such that $\vartheta(0)=0$ and
    \begin{align*}
        \vartheta(t) := \int_{0}^{t}b(s)\theta(s)ds\quad\forall t\in[0,T],
    \end{align*}
    therefore, $\dot{\vartheta}(t) = b(t)\theta(t)$ for a.e. $t\in[0,T]$. Since $b$ is a nonnegative function, it follows that:
    \begin{align*}
        \dot{\vartheta}(t)
        &\leq b(t)\left(\theta(t_{0})+a(t)+\int_{0}^{t}b(s)\theta(s)ds\right) \\
        &\leq b(t)(\theta(0)+a(t)) + b(t)\vartheta(t)\qquad\qquad\textrm{a.e. }t\in[0,T].
    \end{align*}
    Therefore, applying the integrating factor $\mu(t):=\exp\left(-\int_{0}^{t}b(s)ds\right)$ to the expression:
    \begin{align*}
        \dot{\vartheta}(t) - b(t)\vartheta(t)\leq (\theta(0)+a(t))b(t) \quad\textrm{a.e. }t\in[0,T],
    \end{align*}
    it follows that:
    \begin{align*}
        \vartheta(t)\mu(t) \leq \int_{0}^{t} \big(\theta(0)+a(s)\big)b(s)\mu(s)ds\quad\forall t\in[0,T],
    \end{align*}
    thus:
    \begin{align*}
        \vartheta(t) \leq \exp\left(\int_{0}^{t}b(s)ds\right)\int_{0}^{t} \big(\theta(0)+a(s)\big)b(s)\exp\left(\int_{0}^{s}b(\sigma)d\sigma\right)ds\quad\forall t\in[0,T],
    \end{align*}
    therefore:
    \begin{align*}
        \theta(t) 
        &\leq \theta(0)+a(t) + \int_{0}^{t} b(s)\theta(s)ds\\
        &\leq \theta(0)+a(t)+\exp\left(\int_{0}^{t}b(s)ds\right)\int_{0}^{t} \big(\theta(0)+a(s)\big)b(s)\exp\left(\int_{0}^{s}b(\sigma)d\sigma\right)ds\quad\forall t\in[0,T].
    \end{align*}
    
\end{proof}

In the following, several results derived from the Banach Fixed-Point Theorem are presented, as they will be used in the proof of the main result. In particular, a class of operators admitting fixed points in spaces of functions is introduced. A relevant subclass is formed by the so-called \emph{history-dependent operators} (see \cite[Chapter 2]{MR3752610}). An operator $R\colon C([0,T];\mathcal{H})\to C([0,T];\mathcal{H})$ is called a history-dependent operator if there exists a constant $L_{R}>0$ such that, for all $t\in[0,T]$, one has:
\begin{equation*}
    \|R(x)(t)-R(y)(t)\|
    \leq L_{R}\int_{0}^{t} \|x(s)-y(s)\|ds,
    \qquad \forall x,y\in C([0,T];\mathcal{H}).
\end{equation*}

This class of operators satisfies an integral-type contractive condition, which allows the application of fixed point arguments in trajectory spaces. Moreover, such operators arise naturally in the modeling of dynamical systems with memory, where the evolution in time depends not only on the current state but also on the past states of the trajectory. Typical examples include integro-differential equations, inclusions with memory, and Volterra-type models.

\begin{proposicion}[History-dependent fixed-point]
    \label{prop:HistoryDependentFixedPoint}
    Let $R\colon C([0,T];\mathcal{H})\to C([0,T];\mathcal{H})$ be a history-dependent operator. Then, the operator $R$ has a unique fixed point.
\end{proposicion}

In what follows, a generalized version of the previous concept is introduced.

\begin{definicion}[Generalized almost history-dependent operator]
    An operator $R\colon C([0,T];\mathcal{H})\to C([0,T];\mathcal{H})$ is called a Generalized Almost History-Dependent Operator if there exist $\ell\in[0,1)$ and $L_{R}\in L^{p}([0,T];\mathbb{R}_{+})$ with $p\in(1,\infty)$, such that for all $t\in[0,T]$, one has
    \begin{eqnarray*}
        \|R(x)(t) - R(y)(t)\| \leq \ell\|x(t)-y(t)\| + \int_{0}^{t} L_{R}(s) \|x(s) - y(s)\| ds\quad\forall x,y\in C([0,T];\mathcal{H}).
    \end{eqnarray*}
\end{definicion}

\begin{proposicion}[Generalized almost history-dependent fixed-point]
    \label{prop:GeneralizedHistoryDependentFixedPoint}
    Let $R\colon C([0,T];\mathcal{H})\to C([0,T];\mathcal{H})$ be an Generalized Almost History-Dependent Operator. Then, $R$ has an unique fixed point.
\end{proposicion}  

\begin{proof}
    For $\gamma\geq0$, the norm $\|\cdot\|_{\infty,\gamma}$ on $C([0,T];\mathcal{H})$ is defined by $\|f\|_{\infty,\gamma} := \sup_{t\in[0,T]} \|e^{-\gamma t}f(t)\|$. This norm is equivalent to $\|\cdot\|_{\infty}$. Let $t\in[0,T]$ and $x,y \in C([0,T];\mathcal{H})$. Then,
    \begin{align*}
        \|R(x)(t) - R(y)(t)\| 
        &\leq 
        \ell\|x(t) - y(t)\| + \int_{0}^{t} L_{R}(s) \|x(s)-y(s)\|ds\qquad\forall t\in[0,T].
    \end{align*}
    Multiplying the previous inequality by $e^{-\gamma t}$, with $\gamma>0$, yields
    \begin{align*}
        e^{-\gamma t} \|R(x)(t) - R(y)(t)\| 
        &\leq 
        e^{-\gamma t}\ell\|x(t) - y(t)\| + e^{-\gamma t}\int_{0}^{t} L_{R}(s)\|x(s)-y(s)\|ds\\
        &\leq 
        \ell\|x-y\|_{\infty,\gamma}
        + e^{-\gamma t}\|x-y\|_{\infty,\gamma}\int_{0}^{t} L_{R}(s)e^{\gamma s}ds\\
        &\leq 
        \|x-y\|_{\infty,\gamma}
        \left(
            \ell + e^{-\gamma t}
            \left(\int_{0}^{t} L_{R}^{p}(s)ds\right)^{1/p}
            \left(\int_{0}^{t} e^{\gamma qs}ds\right)^{1/q}
        \right)\\
        &= 
        \|x-y\|_{\infty,\gamma}
        \left(
            \ell + e^{-\gamma t}
            \left(\int_{0}^{t} L_{R}^{p}(s)ds\right)^{1/p}
            \left(\frac{e^{\gamma qt}-1}{\gamma q}\right)^{1/q}
        \right)\\
        &=
        \|x-y\|_{\infty,\gamma}
        \left(
            \ell +
            \left(\int_{0}^{t} L_{R}^{p}(s)ds\right)^{1/p}
            \left(\frac{1-e^{-\gamma qt}}{\gamma q}\right)^{1/q}
        \right)\\
        &\leq
        \|x-y\|_{\infty,\gamma}
        \left(
            \ell + \|L_{R}\|_{L^{p}([0,T])}
            \left(\frac{1}{\gamma q}\right)^{1/q}
        \right)
        \qquad\forall t\in[0,T],
    \end{align*}
    where $1<p,q<\infty$ are Hölder conjugates. Consequently,
    \begin{align*}
        \|R(x)-R(y)\|_{\infty,\gamma}
        &\leq
        \left(
            \ell + \|L_{R}\|_{L^{p}([0,T])}
            \left(\frac{1}{\gamma q}\right)^{1/q}
        \right)\|x-y\|_{\infty,\gamma}.
    \end{align*}
    By choosing
    \begin{align*}
        \gamma>\frac{\|L_{R}\|_{L^{p}([0,T])}^{q}}{q(1-\ell)^{q}},
    \end{align*}
    it follows that
    \begin{align*}
        \ell + \|L_{R}\|_{L^{p}([0,T])}
        \left(\frac{1}{\gamma q}\right)^{1/q}<1.
    \end{align*}
    Therefore, $R$ is a contraction on the Banach space $(C([0,T];\mathcal{H});\|\cdot\|_{\infty,\gamma})$. Hence, the existence and uniqueness follow from Banach's Fixed Point Theorem (see \cite[Theorem 3.48]{MR2378491}).
\end{proof}

A fixed-point result for multi-valued mappings is now presented. This result will play a key role in the analysis of the multi-valued dynamics associated with differential inclusions.

\begin{proposicion}[Nadler fixed-point]
    \label{prop:NadlerFixedPoint}
    Let $(X,d)$ be a complete metric space and let 
    $S\colon X\rightrightarrows X$ be a multivalued mapping with nonempty, closed and bounded values. Assume that there exists $L_{S}\in[0,1)$ such that:
    \begin{equation*}
        \operatorname{Hauss}(S(x),S(y))\leq L_{S} d(x,y)
        \qquad\forall x,y\in X,
    \end{equation*}
    where $\operatorname{Hauss}$ denotes the Hausdorff distance induced by $d$. Then $S$ admits a fixed point, that is, there exists $\overline{x}\in X$ such that $\overline{x}\in S(\overline{x})$.
\end{proposicion}

For further details on these fixed-point results, the reader is referred to \cite{MR3752610} and \cite{MR254828}, where the proofs of Propositions \ref{prop:HistoryDependentFixedPoint} and \ref{prop:NadlerFixedPoint} can be found, respectively. This section is concluded with the following compactness result.

\begin{lema}[Compactness results, \cite{MR3626639}]
    \label{compactness}
     Let $(x_n)_n$ be a sequence of absolutely continuous functions from $[0,T]$ into $\mathcal{H}$ with $x_n(0)=x_0$.  Assume that for all $n\in \mathbb{N}$
    \begin{equation}\label{acotamiento}
        \begin{aligned}
            \Vert \dot{x}_n(t)\Vert &\leq \psi(t) \qquad\textrm{ a.e. } t\in [0,T],
        \end{aligned}
    \end{equation}
    where $\psi\in L^1(0,T)$ and that $x_0^n \to x_0$ as $n\to \infty$. Then, there exists a subsequence $(x_{n_{k}})_{k}$ of $(x_{n})_{n}$ and $x\in W^{1,1}([0,T];\mathcal{H})$ such that
    \begin{enumerate}
        \item  $x_{n_{k}}(t)\rightharpoonup x(t)$ in $\mathcal{H}$ as $k\to +\infty$ for all $t\in [0,T]$.
        \item $x_{n_{k}}\rightharpoonup x$ in $L^{1}\left([0,T];\mathcal{H}\right)$ as $k\to +\infty$.
        \item $\dot{x}_{n_{k}}\rightharpoonup \dot{x}$ in $L^{1}\left([0,T];\mathcal{H}\right)$ as $k\to +\infty$.
        \item $\|\dot{x}(t)\| \leq \psi(t)$ a.e. $t\in [0,T]$.
    \end{enumerate}
\end{lema}

\section{Standing Assumptions and Some Consequences}\label{section:DynamicalSystemAndStandingAssumptions}

 The following class of maximal monotone differential inclusions is considered:
\begin{align}
    \label{eq:DynamicWithNonLocalInitialCondition}
    \begin{cases}
        \displaystyle\dot{x}(t) \in -A(t)x(t) + F(t,x(t)) + \int_{0}^{t} g(t,s,x(s))ds \quad\textrm{a.e. }t\in[0,T],\\
        x(0) = h(x(\cdot)) \in \overline{D(A(0))}. 
    \end{cases}
\end{align}
where for each $\in [0,T]$, $A(t)\colon D(A(t))\subset \mathcal{H}\rightrightarrows\mathcal{H}$ is a set-valued maximal monotone operator, $F\colon[0,T]\times\mathcal{H}\rightrightarrows\mathcal{H}$ is a multivalued mapping, $g\colon[0,T]\times[0,T]\times\mathcal{H}\to\mathcal{H}$ is a single-valued mapping such that modeling a process with memory, and $h\colon C([0,T];\mathcal{H}) \longrightarrow\overline{D(A(0))}$ is a mapping. The following assumptions on the data are considered throughout this article:

\begin{enumerate} 

    \item[$(\mathcal{H}_{A})$] The operator $A$ satisfies the following hypotheses:
    \begin{enumerate}[1.]
        \item {\it Vladimirov's absolute continuity condition:}
        \begin{eqnarray}
            \label{eq:vlad}
           \operatorname{dis}(A(t), A(s))\leq \vert a(t)-a(s)\vert \quad \forall t, s \in [0, T], 
        \end{eqnarray}
        for some absolutely continuous function $a\colon[0,T]\to\mathbb{R}$. Here $\operatorname{dis}(A,B)$ denotes the Vladimirov's (see \cite{MR1092799}, \cite{MR1124122}) pseudo-distance defined on the class of maximal monotone operators by:
        \begin{align*}
            \operatorname{dis}(A_{1},A_{2}) := \sup\left\{\frac{\langle y_{1}-y_{2},x_{2}-x_{1}\rangle}{1+\|y_{1}\|+\|y_{2}\|} : (x_{1},y_{1})\in\operatorname{gr}(A_{1}),\ (x_{2},y_{2})\in\operatorname{gr}(A_{2})\right\}.
        \end{align*}
        \item {\it Linear growth condition:}  There exists $c(\cdot)\in L^{1}([0,T];\mathbb{R}_{+})$ such that for a.e. $t\in[0,T]$,
        \begin{eqnarray*}
            \label{eq:LGCT}
            d(0, A(t)(x)) \leq c(t)(1+\Vert x\Vert) \quad \forall  x\in D(A(t)).  
        \end{eqnarray*}
    \end{enumerate}
    
    \item[$(\mathcal{H}_{F})$] The multivalued mapping $F\colon[0,T]\times\mathcal{H}\rightrightarrows\mathcal{H}$ has nonempty, convex and weakly compact values in $\mathcal{H}$, and satisfies the next two conditions:
    \begin{enumerate}[1.]
    \item The set-valued map $t\rightrightarrows \hbox{Gph}F(t,\cdot)$ is measurable.
        \item For a.e. $t\in [0,T]$, $\operatorname{Gph}F(t,\cdot)$ is norm$\times$weak-closed.
    \end{enumerate}
\end{enumerate}

\begin{enumerate}
    \item[$(\mathcal{H}_{F}^{3})$] \textit{Linear growth condition:} There exists $l_{F}\in L^{1}([0,T])$ such that:
    \begin{equation*}
        \|F(t,x)\|_{\max} := \sup\{\|y\| : y \in F(t,x)\} \leq l_{F}(t)(1+\|x\|) \qquad\textrm{a.e. }t\in[0,T],\ \forall x\in\overline{D(A(t))}.
    \end{equation*}
    
    \item[$(\mathcal{H}_{F}^{4})$] \textit{OSL condition:} For a.e. $t\in[0,T]$, $x\rightrightarrows F(t,x)$ satisfies the one-sided Lipschitz condition with constant $L_{F}\in L^{1}([0,T];\mathbb{R}_{+})$, that is:
    \begin{equation*}
        H_{F}(t,x,x-y) - H_{F}(t,y,x-y) \leq L_{F}(t) \|x-y\|^{2} \quad\textrm{a.e. }t\in[0,T],\ \forall x,y\in\overline{D(A(t))}.
    \end{equation*}
    
    \item[$(\mathcal{H}_{g})$] The mapping $g\colon[0,T]\times[0,T]\times\mathcal{H}\to\mathcal{H}$ satisfies the next conditions:
    \begin{enumerate}[1.]
        \item The function $(t,s)\in[0,T]^{2}\mapsto g(t,s,x)$ is measurable for each $x\in\mathcal{H}$.
        \item There exists $L_{g}\in L^{2}(D)$ such that for a.e. $(t,s)\in D$:
        \begin{equation*}
            \|g(t,s,x)- g(t,s,y)\| \leq L_{g}(t,s)\|x-y\|\qquad\forall x, y\in \overline{D(A(s))},
        \end{equation*}
        where $D:=\{(t,s)\in[0,T]^{2}:0\leq s\leq t\leq T\}$.
        \item \textit{Linear growth condition:} There exists $l_{g}\in L^{1}(D)$ such that for a.e. $(t,s)\in D$:
        \begin{equation*}
            \|g(t,s,x)\|\leq l_{g}(t,s)(1+\|x\|)\qquad\forall x\in\overline{D(A(s))}.
        \end{equation*}
    \end{enumerate}

    \item[$(\mathcal{H}_{h})$] The mapping $h\colon C([0,T], \mathcal{H}) \longrightarrow \overline{D(A(0))}$ is $L_{h}$-Lipchitz.
\end{enumerate}

\begin{ejemplo}
    Some standard examples of nonlocal initial conditions satisfying $(\mathcal{H}_{h})$ are presented below:
    \begin{enumerate}
        \item $h(x(\cdot)) = x_{0}$ is the general Cauchy condition, for some $x_{0}\in \overline{D(A)}$.
        \item $h(x(\cdot)) = \pm x(T)$ is the periodic/antiperiodic initial condition.
        \item $h(x(\cdot)) = \frac{1}{T}\int_{0}^{T} x(s)ds$ is the mean value initial condition.
        \item $h(x(\cdot))=\sum_{i=1}^{n} \alpha_{i}x(t_{i})$ with $\alpha_{1},\dots,\alpha_{n}\in\mathbb{R}$ such that $\sum_{i=1}^{n}|\alpha_{i}|\leq 1$ and $0\leq t_{1}<\cdots<t_{n}\leq T$, is the multipoint initial condition.
    \end{enumerate}
\end{ejemplo}



\paragraph{Consequences of $(\mathcal{H}_{A})$:} Some results that follow from the hypothesis $(\mathcal{H}_{A})$ are presented below. 




The hypotheses on the operator $A$ ensure the following existence result. 
\begin{teorema}[Kunze and Monteiro-Marquez, \cite{MR1451848}]
    \label{Kunze1}
    Let $A(t)\colon D(A(t))\subseteq\mathcal{H}\rightrightarrows\mathcal{H}$ be a maximal monotone operator for all $t\in[0,T]$.  Then, under $(\mathcal{H}_{A})$, for all $x_{0}\in \overline{D(A(0))}$, there exists a unique absolutely continuous function $x\colon[0,T]\to\mathcal{H}$ satisfying:
    \begin{align}\label{eq:evolution-equation}
        \begin{cases}
            \dot{x}(t)\in -A(t)x(t) \quad\textrm{a.e. }t\in[0,T],\\
            x(0)=x_{0}.
        \end{cases}
    \end{align}
    Moreover, there exists a constant $K>0$, depending (continuously) only on the parameters $T$, $x_0$, $c(\cdot) \in L^{1}([0,T];\mathbb{R})$ and $\dot{a}(\cdot)\in L^{1}([0,T];\mathbb{R})$, such that
    \begin{equation}
        \label{eq:estimex}
        \Vert x(t)-x(s)\Vert
        \leq K\int_{s}^{t} \big( c(\tau) + |\dot{a}(\tau)|\big) d\tau,
        \qquad 0\leq s\leq t\leq T.
    \end{equation}
    Such a solution will be denoted by $x(x_{0},f)(\cdot)$.
\end{teorema}

\begin{remark} 
    We must emphasize that the estimate \eqref{eq:estimex} does not appear  in Theorem 3  of \cite{MR1451848}, but it is contained in their proof which was given for $c(t)=c$ for all $t\in [0, T]$. We will give a proof in the Appendix for $c(\cdot)\in L^{1}$.    
\end{remark}

In fact, the authors \cite{MR1451848} established Theorem \ref{Kunze1} with $L^{1}$-perturbation dependent only on $t$ without stating the estimate similarly to \eqref{eq:estimex}. The following result is equivalent to that of Theorem \ref{Kunze1}.

\begin{corolario}
    \label{Kunze3}  
    Let $A(t)\colon D(A(t))\subseteq\mathcal{H}\rightrightarrows\mathcal{H}$ be a maximal monotone operator for every $t\in[0,T]$.  Then, under $(\mathcal{H}_{A})$, for every $f(\cdot)\in L^{1}([0,T];\mathcal{H})$ and  every $x_{0}\in \overline{D(A(0))}$, there exists a unique absolutely continuous function $x\colon[0,T]\to\mathcal{H}$ satisfying:
    \begin{align*}
        \begin{cases}
            \dot{x}(t)\in -A(t)x(t)+f(t)\quad\textrm{a.e. }t\in[0,T],\\
            x(0)=x_{0}.
        \end{cases}
    \end{align*}
    Moreover, there exists a constant $K>0$, depending (continuously) only on $T$, $x_{0}$ and $\dot{a}(\cdot),c(\cdot)\in L^1([0,T];\mathbb{R})$, such that:
    \begin{equation*}
        \label{eq:estimex2}
        \|x(t)-x(s)\|
        \leq K\big(1 + \|f\|_{L^{1}([0,T])}\big) \int_{s}^{t} \big(c(\tau)+\Vert \dot{a}(\tau)\Vert+ \Vert f(\tau)\Vert \big) d\tau, \qquad 0\leq s\leq t\leq T.
    \end{equation*} 
\end{corolario}

\begin{proof} 
    It suffices to consider, for each $t\in [0,T]$, the maximal monotone operator $B(t)\colon D(B(t))\subset \mathcal{H}\rightrightarrows \mathcal{H}$ defined by $B(t)x = A(t)\Big(x+\int_{0}^{t} f(s) ds\Big)$, where $D(B(t)) := D(A(t))-\int_{0}^{t} f(s)ds$, and to see that for all $s,t\in [0, T] $, with $s\leq t$,  and $x\in D(B(t))$, one has :
    \begin{align*}
        \operatorname{dis}(B(t),B(s))\leq  \operatorname{dis}(A(t), A(s)) + \int_s^t\Vert f(\tau)\Vert d\tau \leq d(t)-d(s),
    \end{align*}
    where $d(t) = \int_0^t(\Vert \dot{a}(\tau)\Vert+ \Vert f(\tau)\Vert)d\tau$ and 
    \begin{align*}
        d(0, B(t)(x)) &
        = d\left(0, A(t)(x+\int_{0}^{t} f(s)ds)\right) \\
        &\leq c(t)\left(1+\|x\| + \int_{0}^{t} \|f(s)\| ds\right)\\
        &\leq (1+ \|f\|_{L^{1}([0,T])})c(t)(1+\Vert x\Vert).
    \end{align*}
    It now remains to apply Theorem \ref{Kunze1} with the operator $B$. 
\end{proof}

\paragraph{Consequences of $(\mathcal{H}_{F})$:} Some technical lemmas that follow from the hypothesis $(\mathcal{H}_{F})$ are presented below and will be used throughout the article. The following result is a technical consequence of Lemma III.39 in \cite{MR467310}, Theorem 2K in \cite{MR512209}, and Lemma 3.1 in \cite{MR4883326}.

\begin{lema}
    \label{lema:OptimalMeasurableSelection}
    Let $\phi\colon [0, T]\times \mathcal{H}\mapsto \mathbb{R}$ be a $\mathcal{T}\otimes\mathcal{B}$-measurable function with is convex continuous in $x\in\mathcal{H}$. Assume that $(\mathcal{H}_{F})$ is satisfied. Then, for every  measurable mapping $u\colon[0,T]\to\mathcal{H}$, the set-valued mappings $t\rightrightarrows F(t, u(t))$ and  $ t\rightrightarrows \operatorname{argmin}\{\phi(t,y):y\in F(t,u(t))\}$ are measurable and hence there exists a measurable selection $f$ of $t\rightrightarrows  F(t,u(t))$ such that:
    \begin{align*}
        \min_{y\in F(t,u(t))}\phi(t,y) = \phi(t, f(t)).
    \end{align*}
    Consequently, there exists a measurable selection $f$ of $t\rightrightarrows  F(t,u(t))$ such that:
    \begin{align*}
        H_{F}(t,u(t),v(t)) = \langle v(t),f(t)\rangle \quad\textrm{for a.e. }t\in[0,T].
    \end{align*}
\end{lema}

\begin{lema}
    Assume that $(\mathcal{H}_{F}^{4})$ is satisfied by $F$. Then, for every $\varepsilon>0$, the multivalued mapping $(t,x)\mapsto G(t,x):=\operatorname{co}F(t,x+\varepsilon\mathbb{B})$ satisfies the OSL condition $(\mathcal{H}_{F}^{4})$ with the same constant (in $L^{1}([0,T])$) as $F$.
\end{lema}

\begin{proof}
    Let $\varepsilon>0$ and $x,y \in \overline{D(A)}$. Given $\delta > 0$, there exists $d \in \varepsilon \mathbb{B}$ such that:
    \begin{align*}
        \sup_{v\in\operatorname{co}F(t,x+\varepsilon\mathbb{B})} \langle x-y,v\rangle \leq \sup_{v\in F(t,x+d)}\langle x-y,v\rangle + \delta\quad\textrm{a.e. }t\in[0,T].
    \end{align*}
    Then, for a.e. $t\in[0,T]$, it follows that
    \begin{align*}
        H_{G}(t,x,x-y)-H_{G}(t,y,x-y)
        &=\sup_{v\in\operatorname{co}F(t,x+\varepsilon\mathbb{B})} \langle x-y,v\rangle -\sup_{v\in\operatorname{co}F(t,y+\varepsilon\mathbb{B})} \langle x-y,v\rangle  \\
        &\leq \sup_{v\in F(t,x+d)}\langle x-y,v\rangle + \delta - \sup_{v\in F(t,y+d)}\langle x-y,v\rangle \\
        &\leq L_{F}(t)\|x-y\|^{2} + \delta.
    \end{align*}
    Finally, the conclusion follows since $\delta>0$ is arbitrary. 
\end{proof}

We end this section with the following closedness result which is an adaptation of the proof of Theorem 1 in \cite{MR755330} to our situation.

\begin{proposicion}
    \label{prop:Closedness} 
    Suppose that $(\mathcal{H}_{F})$ holds. Let $u_{n},v_{n}\colon[0,T]\mapsto \mathcal{H}$ for all $n\in\mathbb{N}$,  be measurable functions such that $(u_{n})_{n}$ converges almost everywhere on $[0,T]$  to a function $u\colon[0,T]\mapsto\mathcal{H}$ and $(v_{n})_{n}$ converges weakly in $L^{1}([0,T];\mathcal{H})$ to $v\colon[0,T]\mapsto\mathcal{H}$. If $v_n (t) \in  F (t, u_n (t))$ , for all $n\in \mathbb{N}$ and almost all $t\in [0, T]$, then $v(t)\in  F (t,u(t))$ for a.e. $t\in[0,T]$.
\end{proposicion}

\section{The Main Result} 

This section is devoted to the presentation of our existence result for the differential inclusion \eqref{eq:DynamicWithNonLocalInitialCondition}. Subsequently, the notions of solution and $\varepsilon-$solution for differential inclusions with set-valued perturbations are introduced. 

\begin{definicion}[Solution]
    Let $A$, $F$, and $g$ be as described in Section \ref{section:DynamicalSystemAndStandingAssumptions}. Consider the dynamics (with a Cauchy initial condition):
    \begin{align}
        \label{eq:DynamicWithLocalInitialCondition}
        \begin{cases}
            \displaystyle\dot{x}(t) \in -A(t)x(t) + F(t,x(t)) + \int_{0}^{t} g(t,s,x(s))ds \quad\textrm{a.e. }t\in[0,T],\\
            x(0) = x_{0} \in \overline{D(A(0))}.
        \end{cases}
    \end{align}
    \begin{enumerate}
        \item A continuous function $x(\cdot)$ is called a solution of \eqref{eq:DynamicWithLocalInitialCondition} if there exists function $f(\cdot)\in L^{1}([0,T];\mathcal{H})$ such that $x(\cdot)=x(x_{0},f)(\cdot)$ (see Theorem \ref{Kunze1} for the definition) and
        \begin{equation*}
            f(t)\in F(t,x(t))+\int_{0}^{t}g(t,s,x(s))ds
            \qquad\textrm{for a.e. }t\in[0,T].
        \end{equation*}
    
        \item Let $\varepsilon>0$. A continuous function $x(\cdot)$ is called an $\varepsilon$-solution of \eqref{eq:DynamicWithLocalInitialCondition} if there exists a function $f(\cdot)\in L^{1}([0,T];\mathcal{H})$ such that $x(\cdot)=x(x_{0},f)(\cdot)$ and
        \begin{equation*}
            f(t)\in F(t,x(t)+\varepsilon\mathbb{B})+\int_{0}^{t}g(t,s,x(s))ds
            \qquad\textrm{for a.e. }t\in[0,T].
        \end{equation*}
    \end{enumerate}
\end{definicion}

\subsection{Existence of $\varepsilon$-Solutions with Memory}

This subsection is dedicated to the  existence of $\varepsilon$-solutions to the Cauchy problem \eqref{eq:DynamicWithLocalInitialCondition}. Subsequently, a uniform boundedness result for the state of this class of solutions is presented, together with a uniform integrability property of the perturbation appearing in their velocities.

\begin{teorema}[Existence of $\varepsilon$-solutions with memory]
    \label{teo:ExistenceOfEpsilonSolutions}
    Assume that $(\mathcal{H}_{A})$, $(\mathcal{H}_{F})$, $(\mathcal{H}_{F}^{3})$, and $(\mathcal{H}_{g})$ hold. Then, for every $\varepsilon>0$, the problem \eqref{eq:DynamicWithLocalInitialCondition} admits an $\varepsilon$-solution.
\end{teorema}

This theorem extends \cite[Lemma 1]{MR1910880} to problem \eqref{eq:DynamicWithLocalInitialCondition} by incorporating a Volterra term. However, the presence of the Volterra term requires more refined arguments and relies on the fact that the differential inclusion
\begin{align*}
    \begin{cases}
        \dot{x}(t)\in -A(t)x(t)+f(t) \quad \textrm{a.e. } t\in[0,T],\\
        x(0)=x_{0}
    \end{cases}
\end{align*}
admits a unique solution $x(x_0, f)$, for every integrable mapping $f\colon[0,T]\to\mathcal{H}$, as guaranteed by Corollary \ref{Kunze3}.

\begin{proof}
    To begin with, it will be shown that it is possible to construct an $\varepsilon$-solution on a subinterval $I\subseteq[0,T]$, and then, by means of Zorn's Lemma, the existence of an $\varepsilon$-solution on $[0,T]$ will be established.

    \noindent$\circ$ \textit{Recurrence:} Let $t_{0}:=0$. We will construct recursively a sequence $(x_{k})_{k}$ of elements in $\mathcal{H}$, an increasing sequence $(t_{k})_{k}$ such that $\pi:=\{t_{k}:k\in\mathbb{N}\}$ is a partition of a subinterval $[0,t_{\infty}]$ of $[0,T]$, where $\displaystyle t_{\infty}:=\lim_{k\to\infty}t_{k}$, and a sequence $(x_{\pi}^{k}(\cdot))_{k\geq1}$ of absolutely continuous functions $x_{\pi}^{k}\colon[t_{k-1},t_{k}]\to\mathcal{H}$ as follows:
    \begin{align*}
        x_{k}:=
        \begin{cases}
            x_{0}                  
            &\textrm{if $k=0$,}\\
            x_{\pi}^{k}(t_{k}) 
            &\textrm{if $k\geq 1$, in which case it is assumed that for every $j=1,\dots,k$, there exist $t_{j}\in[0,T]$}\\
            &\textrm{and $x_{\pi}^{j}(\cdot)\in W^{1,1}([t_{j-1},T];\mathcal{H})$, such that $x_{\pi}^{j}(\cdot)$ is the unique solution to $(P_{j})$, given by:}\\
            &(P_{j}) \begin{cases}
                \displaystyle
                \dot{x}(t)\in -A(t)x(t)+f_{j}(t)
                +\int_{t_{j-1}}^{t}g(t,s,x(s))ds\\
                \qquad\qquad\qquad\qquad\qquad\quad\displaystyle
                +\sum_{i=1}^{j-1}\int_{t_{i-1}}^{t_{i}}g(t,s,x_{\pi}^{i}(s))ds
                \quad \textrm{a.e. }t\in[t_{j-1},T],\\
                x(t_{j-1})=x_{j-1},
            \end{cases}\\
            &t_{j} := \sup\{t\geq t_{j-1}:\|x_{\pi}^{j}(t)-x_{k-1}\|<\varepsilon\},   \\
            &\textrm{and $f_{j}\in L^{1}([t_{j-1},T];\mathcal{H})$ is defined by $f_{j}(t) = \operatorname{Proj}_{F(t,x_{j-1})}(0)$ for a.e. $t\in[t_{j-1},T]$.}
        \end{cases}
    \end{align*}
    Of course the sum $\displaystyle\sum_{i=1}^{j-1}\int_{t_{i-1}}^{t_{i}} g(t,s,x_{\pi}^{i}(s))ds$ will be zero for $j=1$. Note that the measurability and the integrability of $f_j$ are guaranteed by Lemma \ref{lema:OptimalMeasurableSelection} and hypothesis $(\mathcal{H}_{F}^{3})$, respectively.  In what follows, it is proved that for every $k\in\mathbb{N}$, the existence of $x_{k-1}$ implies the existence of $x_{\pi}^{k}(\cdot)$ as the unique solution of $(P_{k})$ according to the recursive construction above. Let $k\in\mathbb{N}^{*}$, the mapping $\varphi\colon C([t_{k-1},T];\mathcal{H})\to C([t_{k-1},T];\mathcal{H})$ is defined by $\varphi(y(\cdot))=x(\cdot)$, where $x(\cdot)$ is the unique solution to the problem:
    \begin{align*}
        (P_{k}^{\varphi})
        \begin{cases}
            \displaystyle\dot{x}(t) \in -A(t)x(t) + f_{k}(t) + \int_{t_{k-1}}^{t} g(t,s,y(s))ds + \sum_{i=1}^{k-1} \int_{t_{i-1}}^{t_{i}} g(t,s,x_{\pi}^{i}(s))ds \quad\textrm{for a.e. }t\in[t_{k-1},T],\\
            x(t_{k-1})=x_{k-1},
        \end{cases}
    \end{align*}
    guaranteed by Corollary \ref{Kunze3}. Moreover, it follows that $x(\cdot)\in  W^{1,1}([t_{k-1},T];\mathcal{H})$. It is claimed that $\varphi$ is well-defined and admits a unique fixed point: Let $y_{1}(\cdot),y_{2}(\cdot)\in C([t_{k-1},T];\mathcal{H})$, and set $x_{1}:=\varphi(y_{1})$ and $x_{2}:=\varphi(y_{2})$. Then, for a.e. $t\in[t_{k-1},T]$, one has:
    \begin{align*}
        -\dot{x}_{1}(t) + f_{k}(t) + \int_{t_{k-1}}^{t}g(t,s,y_{1}(s))ds + \sum_{i=1}^{k-1} \int_{t_{i-1}}^{t_{i}} g(t,s,x_{\pi}^{i}(s))ds\in A(t)x_{1}(t),\\
        -\dot{x}_{2}(t) + f_{k}(t) + \int_{t_{k-1}}^{t}g(t,s,y_{2}(s))ds + \sum_{i=1}^{k-1} \int_{t_{i-1}}^{t_{i}} g(t,s,x_{\pi}^{i}(s))ds\in A(t)x_{2}(t),
    \end{align*}
    Then:
    \begin{align*}
        \left\langle -\dot{x}_{1}(t) + f_{k}(t) + \int_{t_{k-1}}^{t}g(t,s,y_{1}(s))ds - \Big(-\dot{x}_{2}(t) + f_{k}(t) + \int_{t_{k-1}}^{t}g(t,s,y_{2}(s))ds\Big),x_{1}(t)-x_{2}(t)\right\rangle\geq0\\\textrm{for a.e. }t\in[t_{k-1},T],
    \end{align*}
    and hence, by $(\mathcal{H}_{g})$, 
    \begin{align*}
        \frac{d}{dt}\frac{\|x_{1}(t)-x_{2}(t)\|^{2}}{2}
        &= \langle \dot{x}_{1}(t)-\dot{x}_{2}(t),x_{1}(t)-x_{2}(t)\rangle \\
        &\leq \left\langle \int_{t_{k-1}}^{t}[g(t,s,y_{1}(s))-g(t,s,y_{2}(s))]ds,x_{1}(t)-x_{2}(t) \right\rangle\\
        &\leq \|x_{1}(t)-x_{2}(t)\| \int_{t_{k-1}}^{t} L_{g}(t,s)\|y_{1}(s)-y_{2}(s)\|ds \qquad\textrm{for a.e. }t\in[t_{k-1},T].
    \end{align*}  
    Set $\phi(t) = \tfrac{1}{2}\|x_{1}(t)-x_{2}(t)\|^{2}$. Then:
    \begin{align*}
        \dot{\phi}(t) \leq \sqrt{2}\sqrt{\phi(t)} \int_{t_{k-1}}^{t} L_{g}(t,s)\|y_{1}(s)-y_{2}(s)\|ds \qquad\textrm{for a.e. }t\in[t_{k-1},T].
    \end{align*}
    So that, dividing by $\sqrt{\phi(t)}$ and integrating both sides of the last inequality, one obtains
    \begin{align*}
        \sqrt{\phi(t)} \leq {\sqrt{2}\over 2}\int_{t_{k-1}}^{t}\left(\int_{t_{k-1}}^{s}L_{g}(s,\sigma)\|y_{1}(\sigma)-y_{2}(\sigma)\| d\sigma\right)ds,
    \end{align*}
    or equivalently
    \begin{align*}
        \|x_{1}(t)-x_{2}(t)\| \leq \int_{t_{k-1}}^{t} \left(\int_{t_{k-1}}^{s} L_{g}(s,\sigma)\|y_{1}(\sigma)-y_{2}(\sigma)\|d\sigma\right) ds \quad\forall t\in[t_{k-1},T].
    \end{align*}
    It follows that:
    \begin{align*}
        \|x_{1}(t)-x_{2}(t)\| 
        \leq \int_{t_{k-1}}^{t} \left(\int_{t_{k-1}}^{T} L_{g}(s,\sigma)ds\right)\|y_{1}(\sigma)-y_{2}(\sigma)\| d\sigma \quad\forall t\in[t_{k-1},T].\\
    \end{align*}
    Since the function $\sigma\in[t_{k-1},T]\mapsto \int_{t_{k-1}}^{T} L_{g}(s,\sigma)ds$ belongs to $L^{2}([t_{k-1},T])$, it follows that $\varphi(\cdot)$ is a generalized almost history-dependent operator with $p=2$. Therefore, by Proposition \ref{prop:GeneralizedHistoryDependentFixedPoint}, there exists a unique fixed point $x_{\pi}^{k}(\cdot)\in C([t_{k-1},T];\mathcal{H})$ of $\varphi(\cdot)$. Consequently:
    \begin{align*}
        &\dot{x}_{\pi}^{k}(t) \in -Ax_{\pi}^{k}(t) + f_{k}(t) + \int_{t_{k-1}}^{t} g(t,s,x_{\pi}^{k}(s))ds + \sum_{i=1}^{k-1} \int_{t_{i-1}}^{t_{i}} g(t,s,x_{\pi}^{i}(s))ds \quad\textrm{a.e. }t\in[t_{k-1},T],\\
        &x_{\pi}^{k}(t_{k-1}) = x_{k-1}.
    \end{align*}
    Moreover, the solution to this problem belongs to $ W^{1,1}([t_{k-1},T];\mathcal{H})$.

    \noindent$\circ$ \textit{Existence of some $\varepsilon$-solution over $I\subseteq[0,T]$:} Let $\pi$, $(x_{k})_{k}$, $(f_{k})_{k}$ and $(x_{\pi}^{k}(\cdot))_{k}$ be generated by the recursive construction above. Then, $x_{\pi}(\cdot)\in W^{1,1}([0,t_{\infty}];\mathcal{H})$ and $f_{\pi}(\cdot)\in L^{1}([0,t_{\infty}];\mathcal{H})$, are defined as:
    \begin{align*}
        \begin{cases}
            \displaystyle x_{\pi}(t) = x_{0}\mathds{1}_{\{t_{0}\}}(t) + \sum_{k\geq1} x_{\pi}^{k}(t)\mathds{1}_{(t_{k-1},t_{k}]}(t) \quad\forall t\leq t_{\infty}.\\
            \displaystyle f_{\pi}(t) = \sum_{k\geq1} f_{k}(t)\mathds{1}_{(t_{k-1},t_{k}]}(t) \quad\textrm{for a.e. } t\leq t_{\infty}.
        \end{cases}
    \end{align*}
    It is clear that $x_{\pi}(\cdot)\in W^{1,1}([0,t_{\infty}];\mathcal{H})$ and $f_{\pi}(\cdot)\in L^{1}([0,t_{\infty}];\mathcal{H})$, and moreover, given $t\in[0,t_{\infty})$, there exists $k\in\mathbb{N}^{*}$ such that $t\in [t_{k-1},t_{k})$, it follows, from the construction of the sequence $(t_{k})_{k\in \mathbb{N}}$,  that:
    \begin{align*}
        f_{\pi}(t) = f_{k}(t) \in F(t,x_{k-1}) \subseteq F(t,x_{\pi}^{k}(t)+\varepsilon\mathbb{B}) \quad\textrm{for a.e. }t\in[t_{k-1},t_{k}],\ \forall k\in\mathbb{N}^{*}.
    \end{align*}
    Then,
    \begin{align*}
        \begin{cases}
            \displaystyle\dot{x}_{\pi}(t) \in -A(t)x_{\pi}(t) + F(t,x_{\pi}(t)+\varepsilon\mathbb{B}) + \int_{0}^{t} g(t,s,x_{\pi}(s))ds\quad\textrm{a.e. }t\in[0,t_{\infty}],\\
            x_{\pi}(0)=x_{0}.
        \end{cases}
    \end{align*}
    Therefore, $x_{\pi}(\cdot)$ is an $\varepsilon$-solution to \eqref{eq:DynamicWithLocalInitialCondition}.
    
    \noindent$\circ$ \textit{Maximal solution:} Consider the partially ordered set $(\mathcal{Z},\preceq)$ given by:
    \begin{align*}
        &\mathcal{Z}:=\left\{
        \begin{matrix}
            (t_{\infty}^{\pi},x_{\pi}(\cdot))\in[0,T]\times W^{1,1}([0,t_{\infty}^{\pi}];\mathcal{H}): &\textrm{$x_{\pi}(\cdot)$ is an $\varepsilon$-solution of \eqref{eq:DynamicWithLocalInitialCondition} over $[0,t_{\infty}^{\pi})$,}\\
            &\textrm{where $\pi$ is a partition of $[0,t_{\infty}^{\pi}]\subseteq[0,T]$.}
        \end{matrix}
        \right\},\\
        &\textrm{and: }(t_{\infty}^{\pi_{1}},x_{\pi_{1}}(\cdot)) \preceq (t_{\infty}^{\pi_{2}},x_{\pi_{2}}(\cdot)) \iff t_{\infty}^{\pi_{1}} \leq t_{\infty}^{\pi_{2}} \quad\textrm{and}\quad x_{\pi_{1}} = x_{\pi_{2}}|_{[0,t_{\infty}^{\pi_{1}}]}.
    \end{align*}
    It is clear, by the previous items,  that $\mathcal{Z}\neq\emptyset$ and $\preceq$ is a partial order on $\mathcal{Z}$. Let $\mathcal{C}\subseteq\mathcal{Z}$ be a totally ordered subset of $\mathcal{Z}$. Define:
    \begin{align*}
        \begin{cases}
            t_{\mathcal{C}}:= \sup\{t_{\infty}\in[0,T]:(t_{\infty},x(\cdot))\in\mathcal{C}\},\\
            \textrm{$x_{\mathcal{C}}\colon[0,t_{\mathcal{C}}]\to\mathcal{H}$ given by $x_{\mathcal{C}}(t) := x(t)$ for all $t\in[0,t_{\infty}]$ if $t\leq t_{\infty}$, where $(t_{\infty},x(\cdot))\in\mathcal{C}$.}
       \end{cases}
    \end{align*}
    It is possible to observe that $x_{\mathcal{C}}(\cdot)$ is continuous on $[0,t_{\mathcal{C}})$, since for every sequence $(t_{n})_{n}$ converging to $\overline{t}\in[0,t_{\mathcal{C}})$, that is, given $\delta>0$, there exists $n_{0}(\delta)\in\mathbb{N}$ such that for $n\geq n_{0}(\delta)$, $t_{n}\in(max(0, \overline{t}-\delta),\overline{t}+\delta)\subseteq[0,t_{\mathcal{C}})$, there exists $(t_{\delta},x_{\delta}(\cdot))\in\mathcal{Z}$ such that:
    \begin{align*}
        \overline{t}+\delta<t_{\delta}<t_{\mathcal{C}}\quad\textrm{and}\quad
        x_{\mathcal{C}}(t) = x_{\delta}(t) \quad\forall t\in[0,t_{\delta}].
    \end{align*}
    Then, since $x_{\delta}(\cdot)\in W^{1,1}([0,t_{\delta}];\mathcal{H})$, one has:
    \begin{align*}
        \lim_{n\to\infty} x_{\mathcal{C}}(t_{n}) = \lim_{n\to\infty} x_{\delta}(t_{n}) = x_{\delta}(\overline{t}) = x_{\mathcal{C}}(\overline{t}),
    \end{align*}
    Therefore, $x_{\mathcal{C}}(\cdot)$ is continuous at $\overline{t}$. Moreover, $\dot{x}_{\mathcal{C}}(t)=\dot{x}(t)$ for a.e. $t\in[0,t_{\delta})$, therefore $x_{\mathcal{C}}(\cdot)$ is locally absolutely continuous in neighborhoods of $\overline{t}\in[0,t_{\mathcal{C}})$ of the form $[0,t_{\delta})$. Hence, $x_{\mathcal{C}}(\cdot)\in W^{1,1}([0,t_{\mathcal{C}});\mathcal{H})$. Naturally, $(t_{\mathcal{C}},x_{\mathcal{C}}(\cdot))$ is an upper bound of $\mathcal{C}$ in $(\mathcal{Z},\preceq)$. By Zorn's Lemma, there exists a maximal element in $\mathcal{Z}$ with respect to the order $\preceq$, denoted by $(t_{\max},x_{\max}(\cdot))\in[0,T]\times W^{1,1}([0,t_{\max}];\mathcal{H})$.

    Finally, it will be shown that a maximal element $(t_{\max},x_{\max}(\cdot))$ satisfies $t_{\max}=T$ and that $x_{\max}(\cdot)$ is an $\varepsilon$-solution. Arguing by contradiction, it is assumed that $t_{\max}<T$. The following problem is defined:
    \begin{align*}
        (P_{\infty})
        \begin{cases}
            \displaystyle\dot{x}(t) \in -A(t)x(t) + f_{\infty}(t) + \int_{t_{\max}}^{t} g(t,s,x(s))ds + \int_{0}^{t_{\max}}g(t,s,x_{\max}(s))ds\quad\textrm{a.e. }t\in[t_{\max},T],\\
            x(t_{\max})=x^{*}:=x_{\max}(t_{\max}),
        \end{cases}
    \end{align*}
    where $f_{\infty}(t)\in F(t,x^{*})$ for a.e. $t\in[t_{\max},T]$ is a strongly measurable selection. Problem $(P_{\infty})$ admits a unique solution $x_{\infty}(\cdot)\in W^{1,1}([t_{\max},T];\mathcal{H})$, guaranteed by a fixed-point argument similar to the one presented above. Moreover, let
    \begin{align*}
        \tau:=\sup\{t\in[t_{\max},T]:\|x_{\infty}(t)-x^{*}\|<\varepsilon\}.
    \end{align*}
    Therefore, it is easy to verify that $(\tau,x(\cdot))\in\mathcal{Z}$ and $(t_{\max},x_{\max}(\cdot))\preceq(\tau,x(\cdot))$, where $\tau>t_{\max}$ and
    \begin{align*}
        x(t) = x_{\max}(t)\mathds{1}_{[0,t_{\max}]}(t) + x_{\infty}(t)\mathds{1}_{(t_{\max},\tau]}(t)\quad\forall t\in[0,\tau],
    \end{align*}
    which contradicts the maximality of $(t_{\max},x_{\max}(\cdot))$ with respect to $\preceq$. 
\end{proof}

\subsection{Uniform Boundedness of the Trajectories}
\begin{teorema}
    \label{teo:BoundsForEpsilonSolutionsAndHisSelections}
    Assume that $(\mathcal{H}_{A})$, $(\mathcal{H}_{F})$, $(\mathcal{H}_{F}^{3})$, and $(\mathcal{H}_{g}^{3})$ hold. Then, there exists $M>0$ and $\mu\in L^{1}([0,T];\mathbb{R}_{+})$, such that for all $\varepsilon\in(0,1)$ and for every $\varepsilon$-solution $x(\cdot)$ of \eqref{eq:DynamicWithLocalInitialCondition}, one has: 
    \begin{align*}
        \|x(t)\|&\leq M \quad\quad\forall t\in[0,T] \\
        \|F(t,x(t)+\varepsilon\mathbb{B})\|_{\max}+\int_{0}^{t}\|g(t,s,x(s))\|ds&\leq \mu(t) \quad\textrm{ for a.e. }t\in[0,T].
    \end{align*}
    More precisely, $M$ depends (continuously) only on $x_0$, $L_F$ and $l_g$ and 
     $$\mu(t)=\left(l_{F}(t)+\int_{0}^{t} l_{g}(t,s)ds\right)(2+M)\quad\textrm{a.e. } t\in[0,T]$$
     and
      \begin{equation}\label{xdot1}
     \Vert \dot{x}(t)\Vert \leq M(1+\Vert \mu\Vert_{L^1})(c(t)+\vert \dot{a}(t)\vert+\mu(t)) \quad \textrm{a.e. } t\in [0, T].
     \end{equation}
\end{teorema}

\begin{proof}
    Let $y(\cdot)\in W^{1,1}([0,T];\mathcal{H})$ be the unique solution (guaranteed by Corollary \ref{Kunze3}) to
    \begin{align*}
        \begin{cases}
            \dot{y}(t) \in -A(t)y(t) \quad\textrm{a.e. }t\in[0,T],\\
            y(0)=x_{0},
        \end{cases}
    \end{align*}
    and $x(\cdot):=x(x_0;f_x)(\cdot)\in W^{1,1}([0,T];\mathcal{H})$ be an $\varepsilon$-solution to \eqref{eq:DynamicWithLocalInitialCondition}, where $f_{x}(t)\in F(t,x(t)+\varepsilon\mathbb{B})+\int_{0}^{t}g(t,s,x(s))ds$ for a.e. $t\in[0,T]$ is the associated measurable selection. Then:
    \begin{equation}\label{xdotincl}
        -\dot{x}(t)+f_{x}(t) \in A(t)x(t) \quad\textrm{for a.e. }t\in[0,T],
    \end{equation}
    by $(\mathcal{H}_{A})$:
    \begin{align*}
        \langle -\dot{x}(t)+f_{x}(t)+\dot{y}(t) ,x(t)-y(t)\rangle \geq 0\quad\textrm{a.e. }t\in[0,T],
    \end{align*}
    then for a.e. $t\in[0,T]$:
    \begin{align*}
        \frac{d}{dt}\frac{\|x(t)-y(t)\|^{2}}{2} & = \langle \dot{x}(t)-\dot{y}(t),x(t)-y(t)\rangle \\
        & \leq \langle f_{x}(t),x(t)-y(t)\rangle \leq \|f_{x}(t)\|\|x(t)-y(t)\|\\
        & \leq \|x(t)-y(t)\|\left(\|F(t,x(t)+\varepsilon\mathbb{B})\|_{\max} + \int_0^t\|g(t,s,x(s))\| ds\right)\\
        & \leq \|x(t)-y(t)\|\left(L_F(t)(1+\varepsilon+\Vert x(t)\Vert)+ \int_0^t l_g(t,s)(1+\Vert x(s)\Vert)ds\right)\\
        &\leq \|x(t)-y(t)\|\Big(L_F(t)(1+\varepsilon+\Vert x(t)-y(t)\Vert+\Vert y(t)\Vert) \\
        &\hskip 4cm + \int_0^t l_g(t,s)(1+\Vert x(s)-y(s)\Vert+\Vert y(s)\Vert)ds\Big)\\
        &\leq \|x(t)-y(t)\|\left(L_{F}(t)\Vert x(t)-y(t)\Vert+ \int_0^t l_g(t,s)\Vert x(s)-y(s)\Vert ds \right.\\
        & \hskip 3cm  \left. +  (L_F(t)(2+\Vert y\Vert_\infty)+(1+\Vert y\Vert_\infty)\int_0^tl_g(t, s)ds\right)
    \end{align*}
    Define de function $\varphi(t) ={1\over 2}\|x(t)-y(t)\|^{2}$. Then
    \begin{align*}
        \dot{\varphi}(t) &\leq  \sqrt{2\varphi(t)}\left(L_F(t)\sqrt{2\varphi(t)}+\int_0^tl_g(t,s)\sqrt{2\varphi(s)}ds+(L_F(t)(2+\Vert y\Vert_\infty)+(1+\Vert y\Vert_\infty)\int_0^tl_g(t, s)ds\right)\\
        &\leq 2L_F(t)\varphi(t)+2\sqrt{\varphi(t)}\int_0^tl_g(t,s)\sqrt{\varphi(s)}ds+\sqrt{\varphi(t)} \sqrt{2}(2+\Vert y\Vert_\infty)\left(L_F(t)+\int_0^tl_g(t, s)ds\right).
    \end{align*}
   Applying Lemma \ref{lem:EnhancedDifferentialGronwallInequality} with $k_1(t) = \sqrt{2}(2+\Vert y\Vert_\infty)\big(L_F(t)+\int_0^tl_g(t, s)ds\big)$, $k_2(t) = 2L_F(t)$, $k_3(t) =2$ and $k_{4}(t,s)= l_{g}(t,s)$, on obtains for all $t\in[0,T]$:
   \begin{align*}
        \varphi(t) 
        &\leq  \varphi(0) \exp\left[ \int_{0}^{t} \left(k_{1}(s) + k_{2}(s) + k_{3}(s) \int_{0}^{s}k_{4}(s,\tau)d\tau\right)\right]\\
        & + \int_{0}^{t} k_{1}(s) \exp\left[\int_{s}^{t} \left(k_{1}(\tau) + k_{2}(\tau) + k_{3}(\tau) \int_{0}^{\tau}k_{4}(\tau,\sigma)d\sigma\right)d\tau\right]ds\\
        &= \int_{0}^{t} k_{1}(s) \exp\left[\int_{s}^{t} \left(k_{1}(\tau) + k_{2}(\tau) + k_{3}(\tau) \int_{0}^{\tau}k_{4}(\tau,\sigma)d\sigma\right)d\tau\right]ds\\
        & \leq \int_{0}^{T} k_{1}(s)ds \exp\left[\int_{0}^{T} \left(k_{1}(\tau) + k_{2}(\tau) + k_{3}(\tau) \int_{0}^{\tau}k_{4}(\tau,\sigma)d\sigma\right)d\tau\right].
    \end{align*}
    Thus, by the triangle inequality:
    \begin{align*}
        \| x\|_{\infty} \leq M_{0}
        := \Vert y\Vert_{\infty} + \sqrt{\int_{0}^{T} k_{1}(s)ds}\exp\left[{1\over 2}\int_{0}^{T} \left(k_{1}(\tau) + k_{2}(\tau) + k_{3}(\tau) \int_{0}^{\tau}k_{4}(\tau,\sigma)d\sigma\right)d\tau\right].
    \end{align*}
    
    In this way, $\|x(t)\| \leq M_{0}$ for all $t\in[0,T]$. Finally, by $(\mathcal{H}_{F}^{3})$ and $(\mathcal{H}_{g}^{3})$, it follows that:
    \begin{align*}
        &\|F(t,x(t)+\varepsilon\mathbb{B})\|_{\max} + \int_{0}^{t} \|g(t,s,x(s))\|ds \\
        &\leq l_{F}(t)(1+\|x(t)\|+\varepsilon)+\int_{0}^{t} l_{g}(t,s)(1+\|x(s)\|)ds\\
        &\leq \left(l_{F}(t)+\int_{0}^{t} l_{g}(t,s)ds\right)(2+M_0)=:\mu_0(t)\quad\textrm{a.e. } t\in[0,T].
    \end{align*}

    It remains now to establish the inequality \eqref{xdot1}. Indeed, the last inequality asserts that the measurable mapping $f_{x}$ is integrable, in fact 
    \begin{equation*}
        \Vert f_{x}\Vert_{L^{1}([0,T];\mathcal{H})} \leq \Vert \mu_{0}\Vert_{L^{1}([0,T];\mathbb{R})}.  
    \end{equation*}
    Corollary \ref{Kunze3} asserts the following differential inclusion:
    \begin{align*}
        \begin{cases}
            \dot{y}(t)\in -A(t)y(t)+f_x(t) \quad\textrm{a.e. } t\in[0,T],\\ 
            y(0)=x_{0},
        \end{cases}
    \end{align*}
    has a unique solution which coincides with $x$ and satisfying 
    \begin{equation*}
        \Vert\dot{y}(t)\Vert \leq K(1+\Vert f_x\Vert_{L^{1}([0,T];\mathcal{H})})(c(t)+\vert \dot{a}\vert + \Vert f_{x}(t)\Vert)\leq K(1+\Vert\mu_{0}\Vert_{L^{1}([0,T];\mathbb{R})})(c(t)+\vert \dot{a}\vert + \mu_0(t)),
    \end{equation*} 
    for some constant $K>0$, depending only on
    $T$, $x_0$, $c(\cdot) \in L^{1}([0,T];\mathbb{R})$ and $\dot{a}(\cdot)\in L^{1}([0,T];\mathbb{R})$. It suffices to set $M:=\max\{M_{0}, K\}$ and $\mu(t):=\left(l_{F}(t)+\int_{0}^{t} l_{g}(t,s)ds\right)(2+M)$ for a.e. $t\in [0, T]$ to complete the proof. 
\end{proof}

\subsection{Existence of Solutions for Nonlocal Problem with Memory}

In this subsection, we present our  existence result for the nonlocal dynamical system \eqref{eq:DynamicWithNonLocalInitialCondition}. 

\begin{teorema}[Existence of solutions for nonlocal problem]
    \label{theo:Existence-MainResult}
    Assume that $(\mathcal{H}_{A})$, $(\mathcal{H}_{F})$, $(\mathcal{H}^{3}_{F})$, $(\mathcal{H}^{4}_{F})$, $(\mathcal{H}_{g})$ and $(\mathcal{H}_{h})$ hold. If
    \begin{eqnarray}
        \label{eq:ContractionConditionForMainTheorem}
        L_{h}\exp\left[\int_{0}^{T} \left( 2L_{F}(s) + 2\int_{0}^{s}L_{g}(s,\sigma)d\sigma\right)ds\right]<1,
    \end{eqnarray}
    then there exists a solution $x$ to \eqref{eq:DynamicWithNonLocalInitialCondition}. Moreover, there exists  $M$ depends (continuously) only on $x_0$, $L_F$ and $l_g$ such that for
    \begin{align*}
        \mu(t)=\left(l_{F}(t)+\int_{0}^{t} l_{g}(t,s)ds\right)(2+M)\quad\textrm{a.e. } t\in[0,T]
    \end{align*}
    the solution $x$ satisfies the inequality
    \begin{equation*}
        \Vert \dot{x}(t)\Vert \leq M(1+\Vert \mu\Vert_{L^{1}([0,T];\mathbb{R})})(c(t)+\vert \dot{a}(t)\vert+\mu(t)) \quad \textrm{a.e. } t\in [0, T].
    \end{equation*}
\end{teorema}

\begin{remark}[Dynamics in Banach spaces] 
    As in Dontchev et al. (see \cite{MR1910880,MR3869577}), the result can be stated in a more general setting, namely in Banach spaces $X$ whose dual $X^{*}$ is uniformly convex. In this case, the duality mapping $J\colon X\rightrightarrows X^{*}$, given by
    \begin{align*}
        J(x) := \{x^{*}\in X^{*}: \langle x^{*},x\rangle = \|x\|_{X}^{2} = \|x^{*}\|_{X^{*}}^{2}\} \quad\forall x\in X,
    \end{align*}
    is single-valued. 
\end{remark}

Before giving a detailed proof of the main theorem, a consequence concerning \textit{convex sweeping processes} is presented. As observed by Vladimirov \cite{MR1124122}, when $A(t)=N_{C(t)}(\cdot)$ for a family of closed convex sets $(C(t))_{t\in[0,T]}$, Vladimirov's absolute continuity condition for the set-valued mapping $t\rightrightarrows A(t)$ is equivalent to the absolute continuity of the set-valued mapping $t\rightrightarrows C(t)$, since:
\begin{align*}
    \operatorname{dis}(A(t),A(s)) = \operatorname{Haus}(C(t),C(s)),
\end{align*} 
where $\operatorname{Haus}(\cdot,\cdot)$ denotes the Hausdorff distance. This result yields the existence of solutions to the following Volterra sweeping process:
\begin{align}
    \label{eq:ConvexSweepingProcess}
    \begin{cases}
        \displaystyle\dot{x}(t) \in -N_{C(t)}(x(t)) + F(t,x(t))+\int_{0}^{t} g(t,s,x(s))ds\quad\textrm{a.e. }t\in[0,T],\\
        x(0)=h(x(\cdot))\in C(0).
    \end{cases}
\end{align}
Indeed, this problem corresponds to the previous theorem with $A(t)=N_{C(t)}$. In this setting, Vladimirov's absolute continuity condition can be expressed as follows: 
\begin{enumerate}
    \item[$(\mathcal{H}_{C})$] Let $C\colon[0,T]\rightrightarrows\mathcal{H}$ be a set-valued mapping with nonempty, closed, and convex values, such that
    \begin{align*}
        \operatorname{Haus}(C(t),C(s)) \leq |a(t)-a(s)|\qquad\forall t,s\in[0,T],
    \end{align*}
    for some absolutely continuous function $a\colon[0,T]\to\mathbb{R}$.
\end{enumerate}

\begin{corolario}[Volterra sweeping process]
    \label{cor:VolterraSweepingProcess}
    Assume that $(\mathcal{H}_{C})$, $(\mathcal{H}_{F})$, $(\mathcal{H}_{F}^{3})$, $(\mathcal{H}_{F}^{4})$, $(\mathcal{H}_{g})$, and $(\mathcal{H}_{h})$ hold, where $(\mathcal{H}_{h})$ is understood with $C(0)$ in place of $\overline{D(A)}$. Assume also that
    \begin{align*}
        L_{h}\exp\left[\int_{0}^{t}\left(L_{F}(s)+\int_{0}^{s}L_{g}(s,\sigma)d\sigma\right)ds\right]<1\qquad\forall t\in[0,T].
    \end{align*}
    Then, there exists a solution $x(\cdot)\in W^{1,1}([0,T];\mathcal{H})$ of \eqref{eq:ConvexSweepingProcess}.
\end{corolario}

\begin{proof}
    For each $t\in[0,T]$, define $A(t):=N_{C(t)}$. Since $C(t)$ is nonempty, closed, and convex, $A(t)$ is maximal monotone and $D(A(t))=C(t)$. Moreover, by $(\mathcal{H}_{C})$ and the identity
    \begin{align*}
        \operatorname{dis}(N_{C(t)},N_{C(s)})=\operatorname{Haus}(C(t),C(s)),
    \end{align*}
    the family $A(t)$ satisfies Vladimirov's absolute continuity condition. Finally, since $0\in N_{C(t)}(x)$ for every $x\in C(t)$, one has
    \begin{align*}
        d(0,A(t)(x))=0\qquad\forall x\in D(A(t)).
    \end{align*}
    Hence, the growth condition in $(\mathcal{H}_{A(t)})$ is satisfied with $l_{A}\equiv0$. Therefore, the non-autonomous version of Theorem \ref{theo:Existence-MainResult} can be applied with $A(t)=N_{C(t)}$, which yields a solution of \eqref{eq:ConvexSweepingProcess}. 
\end{proof}

If $A(t)=A$ for every $t\in[0,T]$, where $A\colon D(A)\rightrightarrows\mathcal{H}$ is a fixed maximal monotone operator, that is, if the part of the dynamics governed by the maximal monotone operator is autonomous, the hypothesis $(\mathcal{H}_{A})$ becomes restrictive. Indeed, there exist maximal monotone operators that do not satisfy, in particular, the condition $(\mathcal{H}_{A}^{2})$. A classical example is given by the Laplace operator defined on $L^{2}(\Omega)$, where $\Omega\subseteq\mathbb{R}^{n}$ is a sufficiently regular open set. In this setting, $(\mathcal{H}_{A}^{2})$ would require a growth estimate for the minimal section of the operator in terms of the $L^{2}$-norm of the state, which is not satisfied in general by the Laplace operator.
The technique introduced below to establish the existence of solutions to \eqref{eq:DynamicWithNonLocalInitialCondition} remains valid in the autonomous setting without requiring the hypothesis $(\mathcal{H}_{A}^{2})$. In this situation, the Brézis-Komura Theorem (see \cite{MR348562}) can be used as the basic existence result for the Cauchy problem
\begin{align*}
    \begin{cases}
        \dot{x}(t)\in -Ax(t)+f(t)\quad\textrm{a.e. }t\in[0,T],\\
        x(0)=x_{0}\in\overline{D(A)},
    \end{cases}
\end{align*}
instead of Theorem \ref{Kunze1} of Kunze and Monteiro Marques. Consequently, the autonomous setting can be treated under weaker assumptions on the maximal monotone operator, while preserving the remaining arguments concerning the perturbation and the nonlocal initial condition. This observation is particularly relevant since a large class of important examples falls within the autonomous framework.

\begin{teorema}[Existence of solutions for nonlocal problem with autonomous operators]
   \label{theo:Existence-MainResult2}
   Let $A: \mathcal{H}\rightrightarrows \mathcal{H}$ be a maximal monotone operator. Assume that  $(\mathcal{H}_{F})$, $(\mathcal{H}^{3}_{F})$, $(\mathcal{H}^{4}_{F})$, $(\mathcal{H}_{g})$ and $(\mathcal{H}_{h})$ hold. If
    \begin{eqnarray*}
        L_{h}\exp\left[\int_{0}^{T} \left( 2L_{F}(s) + 2\int_{0}^{s}L_{g}(s,\sigma)d\sigma\right)ds\right]<1,
    \end{eqnarray*}
    then there exists a solution $x$ to the following differential inclusion
    \begin{align*}
        \begin{cases}
            \dot{x}(t) \in  -Ax(t) + F(t,x(t)) +  \displaystyle\int_{0}^{t} g(t,s,x(s))ds,\\
            x(0)=h(x(\cdot))\in\overline{D(A)}.    
        \end{cases}
    \end{align*}
\end{teorema}

\subsection{Proof of the Existence of Solutions for Nonlocal Problem with Memory}

The proof is divided into three Lemmas. The following lemma provides a pointwise estimate between an $\varepsilon$-solution and a $\delta$-solution, for any $\delta<\varepsilon$, where the one-sided Lipschitz property of $F$ plays a fundamental role.

\begin{lema}
    \label{lem:ExistenceResultPart1}
    Assume that $(\mathcal{H}_{A})$, $(\mathcal{H}_{F})$, $(\mathcal{H}^{3}_{F})$, $(\mathcal{H}_{F}^{4})$ and $(\mathcal{H}_{g})$ hold. Pick $x_{0},y_{0}\in \overline{D(A)}$. Then, there exists $C_{1}>0$ such that for all $\varepsilon\in(0,1)$, for all $\varepsilon$-solutions $x(\cdot)\in W^{1,1}([0,T];\mathcal{H})$ of \eqref{eq:DynamicWithLocalInitialCondition} with $x(0)=x_{0}$ and for all $\delta\in(0,\varepsilon)$, there exists $\delta$-solution $y(\cdot)\in W^{1,1}([0,T];\mathcal{H})$ of \eqref{eq:DynamicWithLocalInitialCondition} with $y(0)=y_{0}$, such that:
    \begin{equation}\label{eq:lem:ExistenceResultPart1}
       \|x(t)-y(t)\|
        \leq \|x_{0}-y_{0}\| \exp\Big(\int_{0}^{t} \big[2(\varepsilon+\delta)L_F(s)+L_F(s)+\int_0^sL_g(s,\tau)d\tau+(\varepsilon+\delta)s\big]ds\Big)+ (\varepsilon+\delta)C_1
    \end{equation}
\end{lema} 

\begin{proof}
    Let $\varepsilon>0$ and $\delta\in(0,\varepsilon)$. Let $x(\cdot)$ be an $\varepsilon$-solution   of \eqref{eq:DynamicWithLocalInitialCondition} such that $x(0)=x_{0}$.   Then, Lemma \ref{lema:OptimalMeasurableSelection} ensures the existence of a measurable, hence integrable by Theorem \ref{teo:BoundsForEpsilonSolutionsAndHisSelections},  selection $f$ satisfying 
    \begin{eqnarray*}
        f_{x}(t) \in F(t,x(t)+w(t)) + \int_{0}^{t} g(t,s,x(s))ds\quad\textrm{for a.e. }t\in[0,T],
    \end{eqnarray*}
    such that $x(\cdot)=x(x_{0},f_{x})(\cdot)$, where $w(\cdot)$ is a  measurable function satisfying $\|w(t)\|<\varepsilon$ for a.e. $t\in[0,T]$. Next, a suitable $\delta$-solution will be constructed. For this purpose, a recursive construction similar to that used in the proof of Theorem \ref{teo:ExistenceOfEpsilonSolutions} will be carried out.
    
    \noindent$\circ$ \textit{Recursive construction of a $\delta$-solution:} Let $t_{0}:=0$. A sequence $(y_{k})_{k}$ of elements in $\mathcal{H}$, an increasing sequence $(t_{k})_{k}$ such that $\pi:=\{t_{k}:k\in\mathbb{N}\}$ is a partition of a subinterval $[0,t_{\infty}]$ of $[0,T]$, where $t_{\infty}:=\lim_{k\to\infty}t_{k}$, and a sequence $(y_{\pi}^{k}(\cdot))_{k\geq1}$ of absolutely continuous functions $y_{\pi}^{k}\colon[t_{k-1},t_{k}]\to\mathcal{H}$ are recursively constructed as follows:
    \begin{align*}
        y_{k}:=
        \begin{cases}
            y_{0}                  
            &\textrm{if $k=0$,}\\
            y_{\pi}^{k}(t_{k}) 
            &\textrm{if $k\geq 1$, in which case it is assumed that for every $j=1,\dots,k$, there exists $t_{j}\in[0,T]$}\\
            &\textrm{and $y_{\pi}^{j}(\cdot)\in W^{1,1}([t_{j-1},t_j];\mathcal{H})$, such that $y_{\pi}^{j}(\cdot)$ is the unique solution to $(Q_{j})$, given by:}\\
            &(Q_{j}) \begin{cases}
                \displaystyle
                \dot{y}(t)\in -A(t)y(t)+v_{j}(t)
                +\int_{t_{j-1}}^{t}g(t,s,y(s))ds\\
                \qquad\qquad\qquad\qquad\qquad\quad\displaystyle
                +\sum_{i=1}^{j-1}\int_{t_{i-1}}^{t_{i}}g(t,s,y_{\pi}^{i}(s))ds
                \quad \textrm{a.e. }t\in[t_{j-1},t_j],\\
                y(t_{j-1})=y_{j-1},
            \end{cases}\\
            &t_{j} := \sup\{t\geq t_{j-1}:\|y_{\pi}^{j}(t)-y_{j-1}\|<\delta\},   \\
            &\textrm{and $v_{j}\in L^{1}([t_{j-1},T];\mathcal{H})$ is an optimal measurable selection of $F(\cdot, y_{j-1})$ satisfying}\\
            &\displaystyle\sup_{v\in F(t,y_{j-1})} \langle x(t)+w(t)-y_{j-1},v\rangle
        = \langle x(t)+w(t)-y_{j-1},v_{j}(t)\rangle \quad\textrm{for a.e. }t\in[t_{j-1},T].
        \end{cases}
    \end{align*}
    It is verified that the sequence $(v_{k}(\cdot),y_{\pi}^{k}(\cdot),t_{k},y_{k})_{k\geq1}$ is well defined. Indeed, let $k\geq1$ and $(t_{k-1},y_{k-1})\in[0,T]\times\mathcal{H}$. It is clear that the set-valued mapping $t\rightrightarrows F(t,y_{k-1})$ satisfies the assumptions of Lemma \ref{lema:OptimalMeasurableSelection}. Since the mapping $t\in[0,T]\mapsto x(t)+w(t)-y_{k-1}\in\mathcal{H}$ is measurable, there exists a measurable (integrable)  selection $v_{k}\colon[t_{k-1},T]\to\mathcal{H}$ such that $v_{k}(t)\in F(t,y_{k-1})$ for a.e. $t\in[t_{k-1},T]$, and
    \begin{equation}\label{vsel}
        \sup_{v\in F(t,y_{k-1})} \langle x(t)+w(t)-y_{k-1},v\rangle =\langle x(t)+w(t)-y_{k-1},v_{k}(t)\rangle\quad\textrm{for a.e. }t\in[t_{k-1},T].
    \end{equation}
    Moreover, by Theorem \ref{Kunze1}, there exists a unique solution $y_{\pi}^{k}(\cdot)\in W^{1,1}([t_{k-1},T];\mathcal{H})$ to $(Q_{k})$. By the absolute continuity of the integral of nonnegative functions, there exists $\overline{t}\in(t_{k-1},T]$ such that
    \begin{align*}
        \|y_{\pi}^{k}(t) - y_{k-1}\| 
        \leq \int_{t_{k-1}}^{t} \|\dot{y}_{\pi}^{k}(s)\|ds 
        < \delta 
        \quad\forall t\in(t_{k-1},\overline{t}].
    \end{align*}
    Hence, the set $\left\{ t\in[t_{k-1},T]:  \|y_{\pi}^{k}(t)-y_{k-1}\| < \delta\} \right\}$ is not empty which ensures the existence in $]0, T]$ of 
    \begin{align*}
        t_{k}
        :=
        \sup\left\{
            t\in[t_{k-1},T]:
            \|y_{\pi}^{k}(t)-y_{k-1}\|
            <
            \delta\}
        \right\} > t_{k-1},
    \end{align*}
    and consequently, $y_{k}:=y_{\pi}^{k}(t_{k})$ is well defined. Therefore, the recursive construction is well defined. For each $k\in\mathbb{N}^{*}$, the measurable mapping $f_{k}\colon[t_{k-1},t_{k}]\to\mathcal{H}$ is defined by
    \begin{align*}
        f_{k}(t) 
        := v_{k}(t)
        + \int_{t_{k-1}}^{t} g(t,s,y_{\pi}^{k}(s))ds
        + \sum_{i=1}^{k-1}\int_{t_{i-1}}^{t_{i}} g(t,s,y_{\pi}^{i}(s))ds
        \quad\textrm{for a.e. }t\in[t_{k-1},t_{k}],
    \end{align*}
    and $f_{y}\colon[0,t_{\infty}]\to\mathcal{H}$ is defined by
    \begin{align*}
        f_{y}(t) := \sum_{k=1}^{\infty} f_{k}(t)\mathds{1}_{(t_{k-1},t_{k}]}(t) \quad\textrm{for a.e. }t\in[0,t_{\infty}].
    \end{align*}
    It follows that the mapping
    \begin{eqnarray}
        \label{eq:ProofLemma-DeltaTrajectory}
        y_{\pi}(t) := y_{0}\mathds{1}_{\{0\}}(t) + \sum_{k=1}^{\infty} y_{\pi}^{k}(t)\mathds{1}_{(t_{k-1},t_{k}]}(t) \quad\forall t\in[0,t_{\infty}],
    \end{eqnarray}
    satisfies $y_{\pi}(\cdot)=y_{\pi}(y_{0},f_{y})(\cdot)$ on the interval $[0,t_{\infty}]$ and is a $\delta$-solution. Indeed, given $t\in[0,t_{\infty})$, there exists $k\geq1$ such that $t\in[t_{k-1},t_{k})$. 

    Arguing similarly to the proof of Theorem \ref{teo:ExistenceOfEpsilonSolutions}, the partially ordered set $(\mathcal{Y},\preceq)$ is defined by:
    \begin{align*}
        &\mathcal{Y}:=\left\{
        \begin{matrix}
            (t_{\infty}^{\pi},y_{\pi}(\cdot))\in[0,T]\times W^{1,1}([0,t_{\infty}^{\pi}];\mathcal{H}): &\textrm{$y_{\pi}(\cdot)$ is an $\delta$-solution of \eqref{eq:DynamicWithLocalInitialCondition} over $[0,t_{\infty}^{\pi})$,}\\
            &\textrm{constructed by the previous recurrence,}\\
            &\textrm{where $\pi$ is a partition of $[0,t_{\infty}^{\pi}]\subseteq[0,T]$.}
        \end{matrix}
        \right\},\\
        &\textrm{and } (t_{\infty}^{\pi_{1}},y_{\pi_{1}}(\cdot)) \preceq (t_{\infty}^{\pi_{2}},y_{\pi_{2}}(\cdot)) \iff t_{\infty}^{\pi_{1}} \leq t_{\infty}^{\pi_{2}} \quad\textrm{and}\quad y_{\pi_{1}} = y_{\pi_{2}}|_{[0,t_{\infty}^{\pi_{1}}]}.
    \end{align*}
    Zorn's Lemma guarantees the existence of a maximal element $(t_{\infty}^{\pi},y_{\pi}(\cdot))\in\mathcal{Y}$ such that $t_{\infty}^{\pi}=T$ and $y_{\pi}(\cdot)\in W^{1,1}([0,T];\mathcal{H})$ is a $\delta$-solution. Moreover, since it is constructed recursively, it follows that for every $t\in[0,T)$, there exists $k\in\mathbb{N}^{*}$ such that $t\in[t_{k-1},t_{k})$. On the other hand, for each $k\in\mathbb{N}^{*}$, the mapping  $v_{k}\in L^{1}([t_{k-1},t_{k}];\mathcal{H})$,  constructed  in (\ref{vsel}), satisfies 
    \begin{align}
        H_{F}(t,y_{k-1},x(t)+w(t)-y_{k-1}) 
        &= \sup_{v\in F(t,y_{k-1})} \langle x(t)+w(t)-y_{k-1},v\rangle \nonumber\\
        &= \langle x(t)+w(t)-y_{k-1},v_{k}(t)\rangle \quad\textrm{for a.e. }t\in[t_{k-1},t_{k}] \label{eq:OSL-Condition-For-Delta-Solution},
    \end{align}
    For the remainder of the proof, for each $k\in\mathbb{N}^{*}$, the measurable mapping $f_{k}\colon[t_{k-1},t_{k}]\to\mathcal{H}$ is defined by
    \begin{align*}
        f_{k}(t) 
        := v_{k}(t)
        + \int_{t_{k-1}}^{t} g(t,s,y_{\pi}^{k}(s))ds
        + \sum_{i=1}^{k-1}\int_{t_{i-1}}^{t_{i}} g(t,s,y_{\pi}^{i}(s))ds
        \quad\textrm{for a.e. }t\in[t_{k-1},t_{k}],
    \end{align*}
    and the mapping $f_{y}\colon[0,T]\to\mathcal{H}$ is defined by:
    \begin{align*}
        f_{y}(t) := \sum_{k=1}^{\infty} f_{k}(t)\mathds{1}_{(t_{k-1},t_{k}]}(t) \quad\textrm{for a.e. }t\in[0,T].
    \end{align*}    
    \noindent$\circ$\textit{Estimation from OSL condition:} The $\delta$-solution is denoted by $y(\cdot):=y_{\pi}(y_{0},f_{y})(\cdot)$, constructed over the partition $\pi=\{t_{k}:k\in\mathbb{N}^{*}\}$ of $[0,T]$. Given $t\in[0,T)$, there exists $k\in\mathbb{N}^{*}$ such that $t\in[t_{k-1},t_{k})$. By $(\mathcal{H}_{F}^{4})$, it follows that:
    \begin{eqnarray}
        \nonumber
        H_{F}(t,x(t)+w(t),x(t)+w(t)-y_{k-1}) - H_{F}(t,y_{k-1},x(t)+w(t)-y_{k-1}) \\
        \leq L_{F}(t)\|x(t)+w(t)-y_{k-1}\|^{2}\quad\textrm{a.e. }t\in[t_{k-1},t_{k}] \label{eq:ProofLemma-OSLInequality}.
    \end{eqnarray}
    It is clear that $f_{x}(\cdot)$ satisfies:
    \begin{align*}
        &H_{F}(t,x(t)+w(t),x(t)+w(t)-y_{k-1}) \\&\geq \langle x(t)+w(t)-y_{k-1},f_{x}(t)-\int_{0}^{t}g(t,s,x(s))ds\rangle\quad\textrm{for a.e. }t\in[t_{k-1},t_{k}].
    \end{align*}
    Since the measurable selection $v_{k}$ satisfies  relation \eqref{eq:OSL-Condition-For-Delta-Solution}, this last one can be written in terms of the measurable mapping $f_{k}$, 
    \begin{align*}
        H_{F}(t,y_{k-1},x(t)+w(t)-y_{k-1}) = \langle x(t)+w(t)-y_{k-1},f_{k}(t)-\int_{0}^{t} g(t,s,y(s))ds\rangle \quad\textrm{a.e. }t\in[t_{k-1},t_{k}].
    \end{align*}
    Therefore, inequality \eqref{eq:ProofLemma-OSLInequality} can be expressed as follows:
    \begin{align*}
        &\langle x(t)+w(t)-y_{k-1},f_{x}(t)-f_{k}(t)\rangle \\ 
        &\leq L_{F}(t)\|x(t)+w(t)-y_{k-1}\|^{2} \\&\qquad+ \left\langle x(t)+w(t)-y_{k-1},\int_{0}^{t}[g(t,s,x(s)) - g(t,s,y(s))]ds\right\rangle \quad\textrm{a.e. }t\in[t_{k-1},t_{k}].
    \end{align*}
    Expanding the previous expression and using the definition of the trajectory $y$, the following inequality is obtained:
    \begin{align*}
        &\langle x(t)-y(t),f_{x}(t)-f_{k}(t)\rangle + \langle w(t)+y(t)-y_{k-1},f_{x}(t)-f_{k}(t)\rangle\\
        &\leq L_{F}(t)\|x(t)+w(t)-y_{k-1}\|^{2} + \|x(t)+w(t)-y_{k-1}\|\int_{0}^{t}L_{g}(t,s)\|x(s)-y(s)\|ds \\
        &\leq  L_{F}(t)\Big(\|x(t)-y(t)\|^{2} +\|w(t)+y(t)-y_{k-1}\|^2 + 2\|x(t)-y(t)\| \|w(t)-y(t)-y_{k-1}\|\Big)+ \\
        &\qquad \qquad(\|x(t)-y(t)\|+\|w(t)+y(t)-y_{k-1}\|) \int_{0}^{t} L_{g}(t,s)\|x(s)-y(s)\|ds\\
        & \leq L_F(t)\Big(\| x(t)-y(t)^2 + (\varepsilon+\delta)^2 + 2(\varepsilon+\delta)\|x(t)-y(t)\|\Big)+ \\
        & \qquad \qquad (\|x(t)-y(t)\|+ \varepsilon+\delta) \int_{0}^{t} L_{g}(t,s)\|x(s)-y(s)\|ds\quad\textrm{a.e. }t\in[t_{k-1},t_{k}] 
    \end{align*}
    It follows that, for every $k\in\mathbb{N}^{*}$,
    \begin{align*}
       \langle x(t)-y(t),f_{x}(t)-f_{k}(t)\rangle &\leq L_F(t)\Big(\| x(t)-y(t)^2 + (\varepsilon+\delta)^2 + 2(\varepsilon+\delta)\|x(t)-y(t)\|\Big)+ \Big(\|x(t)-y(t)\|+ \\
        & \qquad  \varepsilon+\delta\Big) \int_{0}^{t} L_{g}(t,s)\|x(s)-y(s)\|ds - \langle w(t)+y(t)-y_{k-1},f_{x}(t)-f_{k}(t)\rangle\\
        &\leq L_F(t)\Big(\| x(t)-y(t)^2 + (\varepsilon+\delta)^2 + 2(\varepsilon+\delta)\|x(t)-y(t)\|\Big)+ \Big(\|x(t)-y(t)\|+ \\
        & \qquad  \varepsilon+\delta\Big) \int_{0}^{t} L_{g}(t,s)\|x(s)-y(s)\|ds+2(\varepsilon+\delta)\mu(t) \quad\textrm{a.e. }t\in[t_{k-1},t_{k}].
       \end{align*}
       The last line of this upper bound is due to Theorem \ref{teo:BoundsForEpsilonSolutionsAndHisSelections}.
    The previous inequality holds for a.e. $t\in[0,T]$ in the following sense:
    \begin{eqnarray}
        &&\langle x(t)-y(t),f_{x}(t) - f_{k}(t)\rangle\mathds{1}_{(t_{k-1},t_{k})}(t) \nonumber\\
        &&\leq \mathds{1}_{(t_{k-1},t_{k})}(t)\Big[L_F(t)\Big(\| x(t)-y(t)^2 + (\varepsilon+\delta)^2 + 2(\varepsilon+\delta)\|x(t)-y(t)\|\Big)+ \Big(\|x(t)-y(t)\|+ \nonumber\\
        && \qquad  \varepsilon+\delta\Big) \int_{0}^{t} L_{g}(t,s)\|x(s)-y(s)\|ds+2(\varepsilon+\delta)\mu(t)\Big] \quad\textrm{a.e. }t\in[0,T]
        \label{eq:ProofLemma-InequalityOverK}
    \end{eqnarray}
    It is observed that, for each $k\in\mathbb{N}^{*}$, one has $y(t)=y_{\pi}^{k}(t)$ for every $t\in[t_{k-1},t_{k}]$ and $f_{y}(t)=f_{k}(t)$ for a.e. $t\in[t_{k-1},t_{k}]$. Therefore, by summing \eqref{eq:ProofLemma-InequalityOverK} over $k$, the following inequality is obtained:
    \begin{align*}
        &\langle x(t)-y(t),f_{x}(t) - f_{y}(t)\rangle\\ 
        &\leq L_F(t)\Big(\| x(t)-y(t)^2 + (\varepsilon+\delta)^2 + 2(\varepsilon+\delta)\|x(t)-y(t)\|\Big)+ \Big(\|x(t)-y(t)\|+ \\
        & \qquad  \varepsilon+\delta\Big) \int_{0}^{t} L_{g}(t,s)\|x(s)-y(s)\|ds+2(\varepsilon+\delta)\mu(t)  \qquad\textrm{for a.e. }t\in[0,T].
    \end{align*}

    \noindent$\circ$ \textit{Estimate based on monotonicity and conclusion:} By $(\mathcal{H}_{A})$ and the fact that
    \begin{align*}
        &\dot{x}(t) \in -A(t)x(t) + f_{x}(t), \\
        &\dot{y}(t) \in -A(t)y(t) + f_{y}(t),
    \end{align*}
    for a.e. $t\in[0,T]$, it holds that
    \begin{align*}
        \langle f_{x}(t)-\dot{x}(t) - f_{y}(t)+\dot{y}(t), x(t)-y(t)\rangle \geq 0 \quad\textrm{a.e. }t\in[0,T],
    \end{align*}
    then:
    \begin{align*}
        \frac{d}{dt}\left(\frac{\|x(t)-y(t)\|^{2} }{2}\right) &\leq \langle x(t)-y(t),f_{x}(t)-f_{y}(t)\rangle \\
        & \leq L_F(t)\Big(\| x(t)-y(t)^2 + (\varepsilon+\delta)^2 + 2(\varepsilon+\delta)\|x(t)-y(t)\|\Big)+ \Big(\|x(t)-y(t)\|+ \\
        & \qquad  \varepsilon+\delta\Big) \int_{0}^{t} L_{g}(t,s)\|x(s)-y(s)\|ds+2(\varepsilon+\delta)\mu(t)  \qquad\textrm{for a.e. }t\in[0,T].
    \end{align*}
    Since the function $\theta\colon[0,T]\to\mathbb{R}_{+}$, defined by $\theta(t):={1\over 2}\|x(t)-y(t)\|^{2}$, is absolutely continuous and satisfies
    \begin{align*}
        \dot{\theta}(t) 
        & \leq 2L_F(t)\theta(t)+4(\varepsilon+\delta)L_F(t) \sqrt{\theta(t)}+ 2\sqrt{\theta(t)}\int_0^tL_g(t,s)\sqrt{\theta(s)}ds+ 2(\varepsilon+\theta)\int_0^tL_g(t,s)\sqrt{\theta(s)}ds+\\
        & \qquad \qquad 2(\varepsilon+\delta)^2L_F(t) + 2(\varepsilon+\delta)\mu(t)\qquad\textrm{for a.e. }t\in[0,T].
    \end{align*}
     Lemma \ref{lem:EnhancedDifferentialGronwallInequality} ensures the following inequalities (with $\varepsilon(t)=2(\varepsilon+\delta)^2L_F(t)+2(\varepsilon+\delta)\mu(t)$, $k_1(t)= 4(\varepsilon+\delta)L_F(t)$, $k_2(t) = 2L_F(t)$, $k_3(t)=1$, $k_4(t,s)=L_g(t,s)$, $k_5(t)=2(\varepsilon+\delta)$ and $k_6(t, s)=L_g(t,s)$)  
     
    \begin{align*}
        \theta(t)
        &\leq \theta(0) \exp\left[\int_{0}^{t} \left(4(\varepsilon+\delta)L_F(s)+2L_F(s)+\int_0^sL_g(s,\tau)d\tau+(\varepsilon+\delta)s\right)ds\right]+\\
        & \qquad \qquad 2(\varepsilon+\delta)\int_{0}^{t}\left((\varepsilon+\delta)L_{F}(s) + \mu(s)+2L_F(s)+\int_{0}^{s} L_{g}^{2}(s,\tau)d\tau \right)E_{\varepsilon, \delta}(t,s)ds
    \end{align*}
    where:
       $$ E_{\varepsilon, \delta}(t,s):=\exp\left[\int_{s}^{t} \left(4(\varepsilon+\delta)L_F(\tau)+2L_F(\tau)+ \int_0^\tau L_g(\tau, \sigma)d\sigma+(\varepsilon+\delta)\tau\right)d\tau\right].$$
    The lemma follows by taking:
    \begin{align*}
        C_{1}:=\sqrt{2\int_{0}^{T}\left(4L_F(s)+\mu(s)+\int_{0}^{s} L_{g}^{2}(s,\tau)d\tau \right)E_{1,1}(t,s)ds}.
    \end{align*}
\end{proof}

\begin{lema}[Filippov-Plis-like Lemma]
    \label{lem:ExistenceResultPart2}
    Assume that $(\mathcal{H}_{A})$, $(\mathcal{H}_{F})$, $(\mathcal{H}^{3}_{F})$, $(\mathcal{H}_{F}^{4})$, and $(\mathcal{H}_{g})$ hold. Then there exists $C_1>$, such that  for all $x_{0}\in \overline{D(A)}$, $\operatorname{Sol}(x_{0})$ is the set of all solutions to \eqref{eq:DynamicWithLocalInitialCondition} satisfying $x(0)=x_{0}$, is nonempty. Moreover, there exists $x\in \operatorname{Sol}(x_{0})$ such that for $\varepsilon\in(0,1)$, $x_{0},y_{0}\in \overline{D(A(0))}$, and $y(\cdot)\in W^{1,1}([0,T];\mathcal{H})$ an $\varepsilon$-solution to \eqref{eq:DynamicWithLocalInitialCondition}, with $y(0)=y_{0}$, one has:
    \begin{align*}
        \|x(t)-y(t)\| \leq \|x_{0}-y_{0}\| \exp\left[\int_{0}^{t} \left(4\varepsilon L_F(s)+L_F(s)+\int_{0}^{s} L_{g}(s,\tau)d\tau+2\varepsilon s\right)ds\right]+ 2\varepsilon C_1 \quad \forall t\in [0, T].
    \end{align*}
    Consequently,
    \begin{align*}
        d_{\operatorname{Sol}(x_{0})}(y(\cdot)) \leq \|x_{0}-y_{0}\| \exp\left[\int_{0}^{T} \left(4\varepsilon L_{F}(s)+L_{F}(s)+\int_{0}^{s}L_{g}(s,\tau)d\tau+2\varepsilon s\right)ds\right]+ 2\varepsilon C_1
    \end{align*}
    where
    \begin{eqnarray*}
        d_{\operatorname{Sol}(x_{0})}(y(\cdot)) := \inf_{x(\cdot)\in\operatorname{Sol}(x_{0})} \|x-y\|_{\infty}.
    \end{eqnarray*}
\end{lema}

\begin{proof}
    Since $y(\cdot)$ is an $\varepsilon$-solution to \eqref{eq:DynamicWithLocalInitialCondition} satisfying $y(0)=y_{0}$ (whose existence is ensured by Theorem \ref{teo:ExistenceOfEpsilonSolutions}), let $\eta>0$ and let $(\varepsilon_{n})_{n}$ be a decreasing sequence of positive numbers converging to $0$, with $\varepsilon_{0}:=\varepsilon$, such that
    \begin{eqnarray*}
        2C_{1}\sum_{n\in\mathbb{N}^{*}} \varepsilon_{n} <\frac{\eta}{2},
    \end{eqnarray*}
    where $C_{1}$ is the constant provided by Lemma \ref{lem:ExistenceResultPart1}, which satisfies inequality \eqref{eq:lem:ExistenceResultPart1} for every $\delta$-solution with $\delta\in(0,\varepsilon)$. Moreover, since $y(\cdot)$ is an $\varepsilon_{0}$-solution, Lemma \ref{lem:ExistenceResultPart1} guarantees the existence of an $\varepsilon_{1}$-solution $x_{1}(\cdot)$ satisfying $x_{1}(0)=x_{0}$, relation (\ref{xdot1}) and the following estimate holds:
    \begin{align*}
        \|x_{1}(t)-y(t)\| 
        &\leq \|x_{0}-y_{0}\| \exp\left[\int_{0}^{t} \left(2(\varepsilon+\delta)L_F(s)+L_F(s)+\int_0^sL_g(s,\tau)d\tau+(\varepsilon+\delta)s\right)ds\right]+ (\varepsilon+\delta)C_{1}.
    \end{align*}
    Proceeding inductively with respect to $n\in\mathbb{N}^{*}$, for every $\varepsilon_{n}$-solution $x_{n}(\cdot)$ satisfying $x_{n}(0)=x_{0}$, Lemma \ref{lem:ExistenceResultPart1} guarantees the existence of an $\varepsilon_{n+1}$-solution $x_{n+1}(\cdot)$ satisfying $x_{n+1}(0)=x_{0}$, relation (\ref{xdot1}) and
    \begin{eqnarray*}
        \|x_{n+1}(t)-x_{n}(t)\| \leq (\varepsilon_{n+1}+\varepsilon_{n})C_{1} \qquad\forall t\in[0,T].
    \end{eqnarray*}
    Consequently, for $m<n$, the sequence of approximate solutions $(x_{n})_{n}$ satisfies
    \begin{align*}
        \|x_{n}(t)-x_{m}(t)\| 
        \leq \sum_{k=m}^{n-1} \|x_{k+1}(t)-x_{k}(t)\| 
        \leq \sum_{k=m}^{n-1} 2C_{1}\varepsilon_{k} :=\delta(n,m)
        \qquad\forall t\in[0,T].
    \end{align*}
    Since $\delta(n,m)$ is independent of $t\in[0,T]$ and satisfies $\delta(n,m)\to 0$ as $n,m\to\infty$, it follows that $(x_{n})_{n}$ is a Cauchy sequence with respect to the uniform norm $\|\cdot\|_{\infty}$. Therefore, there exists $x(\cdot)\in C([0,T];\mathcal{H})$ such that $x_{n}\to x$ uniformly on $[0,T]$. Moreover, for every $n\in\mathbb{N}$, there exists a measurable function $f_{n}$ such that
    \begin{eqnarray*}
        f_{n}(t)\in F(t,x_{n}(t)+\varepsilon_{n}\mathbb{B})+\int_{0}^{t}g(t,s,x_{n}(s))ds
    \quad \textrm{a.e. } t\in[0,T].
    \end{eqnarray*}
    Note that, Theorem \ref{teo:BoundsForEpsilonSolutionsAndHisSelections} ensures
    \begin{align}
       &\|x_{n}\|_{\infty} \leq M &\forall n\in\mathbb{N},\nonumber\\
       &\|f_{n}(t)\| \leq \mu(t)  &\textrm{ a.e. } t\in[0,T],\quad       \forall n\in \mathbb{N},\nonumber\\
       & \label{estimxdot}\Vert \dot{x}_n(t)\Vert \leq M(1+\Vert \mu\Vert_{L^{1}([0,T];\mathbb{R})})\big(c(t)+\vert \dot{a}(t)\vert+\mu(t)\big) &\textrm{a.e. } t\in [0, T],        \forall n\in \mathbb{N},
    \end{align} 
    where $M$ and $\mu(\cdot)$ are given by Theorem \ref{teo:BoundsForEpsilonSolutionsAndHisSelections}. Hence, $(\dot{x}_{n})_{n}$ and $(f_{n})_{n}$ are bounded in $L^{1}([0,T];\mathcal{H})$ and uniformly integrable. By the Dunford-Pettis Theorem (see Theorem 2.3.24 \cite{MR2168068}), there exists  subsequences $(\dot{x}_{n_{k}})_{k}$ and  $(f_{n_{k}})_{k}$ converging weakly in $L^{1}([0,T];\mathcal{H})$ to $\dot x\in L^{1}([0,T];\mathcal{H})$ (by Lemma \ref{compactness} and the uniform convergence of $(x_n)_n$ to $x$) and to some function $f\in L^{1}([0,T];\mathcal{H})$, respectively. Consequently
        \begin{align*}
            \begin{cases}
                \dot{x}(t) \in -A(t)x(t) + f(t) \quad\textrm{a.e. }t\in[0,T],\\
                x(0)=x_{0},
            \end{cases}
        \end{align*}
    that is, $x(\cdot)=x(x_{0},f)(\cdot)$, which satisfies, by relation (\ref{estimxdot}), the inequality
    \begin{align*}
        \Vert \dot{x}(t)\Vert \leq M(1+\Vert \mu\Vert_{L^{1}([0,T];\mathbb{R})})(c(t)+\vert \dot{a}(t)\vert+\mu(t)) \quad \textrm{a.e. } t\in [0, T].
    \end{align*}
    Proposition \ref{prop:Closedness} ensures that $f(t)\in F(t,x(t))+\int_{0}^{t}g(t,s,x(s))ds$ for a.e. $t\in[0,T]$,  thus ensuring the nonemptiness of $\operatorname{Sol}(x_{0})$.  Finally, let $n\in\mathbb{N}$ be such that $\|x-x_{n}\|_{\infty}<\eta/2$. Then, for every $t\in[0,T]$,
    \begin{align*}
        \|x(t)-y(t)\| 
        &\leq \|x(t)-x_{n}(t)\| + \|x_{n}(t)-y(t)\| \\
        &\leq \frac{\eta}{2} + \|x_{1}(t)-y(t)\| + \sum_{k=1}^{n-1} \|x_{k+1}(t)-x_{k}(t)\| \\
        &\leq \frac{\eta}{2} + \|x_{0}-y_{0}\| \exp\left[\int_{0}^{t} \left(2(\varepsilon+\delta)L_F(s)+L_F(s)+\int_0^sL_g(s,\tau)d\tau+(\varepsilon+\delta)s\right)ds\right]+ \\
        &  \hskip 5cm (\varepsilon+\delta)C_1 + 2C_{1}\sum_{k=1}^{n-1}\varepsilon_{k}\\
        &\leq \eta + \|x_{0}-y_{0}\| \exp\left[\int_{0}^{t} \left(2(\varepsilon+\delta)L_{F}(s)+L_{F}(s)+\int_{0}^{s} L_{g}(s,\tau)d\tau+(\varepsilon+\delta)s\right)ds\right]+ \\
        &  \hskip 5cm (\varepsilon+\delta)C_1.
    \end{align*}
    As $\eta>0$ is arbitrary, one obtains:
    \begin{align*}
        &\|x(t)-y(t)\| \\
        &\leq \|x_{0}-y_{0}\| \exp\Big(\int_{0}^{t} \big[2(\varepsilon+\delta)L_F(s)+L_F(s)+\int_0^sL_g(s,\tau)d\tau+(\varepsilon+\delta)s\big]ds\Big)+ (\varepsilon+\delta)C_1 \quad \forall t\in [0, T].
    \end{align*}
\end{proof}

Since each solution to \eqref{eq:DynamicWithLocalInitialCondition} is also an $\varepsilon-$ solution to \eqref{eq:DynamicWithLocalInitialCondition} (for all $\epsilon\in (0,1)$), the following result is a direct consequence of Lemma \ref{lem:ExistenceResultPart2}.

\begin{lema}
    \label{lem:ExistenceResultPart3} 
    Assume that $(\mathcal{H}_{A})$, $(\mathcal{H}_{F})$, $(\mathcal{H}^{3}_{F})$, $(\mathcal{H}_{F}^{4})$ and $(\mathcal{H}_{g})$ hold. For any $x_{0},y_{0}\in\overline{D(A)}$:
    \begin{eqnarray*}
        \operatorname{Haus}(\operatorname{Sol}(x_{0}),\operatorname{Sol}(y_{0})) \leq \|x_{0}-y_{0}\| \exp\left[\int_{0}^{T} \left(L_{F}(s)+\int_{0}^{s}L_{g}(s,\tau)d\tau\right)ds\right]
    \end{eqnarray*}
\end{lema}

\begin{proof}
    By Lemma \ref{lem:ExistenceResultPart2}, $\operatorname{Sol}(x_{0})$ and $\operatorname{Sol}(y_{0})$ are nonempty. Therefore, it is enough to show that
    \begin{eqnarray}
        \label{eq:Lemma9-OneSided}
        \sup_{x(\cdot)\in\operatorname{Sol}(x_{0})} d_{\operatorname{Sol}(y_{0})}(x(\cdot)) \leq \|x_{0}-y_{0}\|\exp\left[\int_{0}^{T}\left(L_{F}(s)+\int_{0}^{s}L_{g}(s,\sigma)d\sigma\right)ds\right],
    \end{eqnarray}
    since interchanging the roles of $x_{0}$ and $y_{0}$ in the argument below then gives, in the same way,
    \begin{eqnarray*}
        \sup_{y(\cdot)\in\operatorname{Sol}(y_{0})} d_{\operatorname{Sol}(x_{0})}(y(\cdot)) \leq \|x_{0}-y_{0}\|\exp\left[\int_{0}^{T}\left(L_{F}(s)+\int_{0}^{s}L_{g}(s,\sigma)d\sigma\right)ds\right],
    \end{eqnarray*}
    and the two bounds together give precisely the stated Hausdorff estimate.\\
    To prove \eqref{eq:Lemma9-OneSided}, let $x(\cdot)\in\operatorname{Sol}(x_{0})$ be arbitrary. Since $x(\cdot)$ is a solution of \eqref{eq:DynamicWithLocalInitialCondition} satisfying $x(0)=x_{0}$, it is, in particular, an $\varepsilon$-solution satisfying $x(0)=x_{0}$, for every $\varepsilon\in(0,1)$. Applying Lemma \ref{lem:ExistenceResultPart2} with the roles of $x_{0}$ and $y_{0}$ interchanged, that is, taking $y_{0}$ as the initial condition whose solution set the distance is measured to, and $x(\cdot)$ as the given $\varepsilon$-solution satisfying $x(0)=x_{0}$ — it follows that, for every $\varepsilon\in(0,1)$,
    \begin{eqnarray*}
        d_{\operatorname{Sol}(y_{0})}(x(\cdot)) \leq \|x_{0}-y_{0}\| \exp\Big(\int_{0}^{T} \big[4\varepsilon L_F(s)+L_F(s)+\int_0^sL_g(s,\tau)d\tau+2\varepsilon s\big]ds\Big)+ 2\varepsilon C_1,
    \end{eqnarray*}
    where $C_{1}>0$ is the constant provided by Lemma \ref{lem:ExistenceResultPart2} (equivalently, by Lemma \ref{lem:ExistenceResultPart1}), which depends only on $T,L_{F}$ and $L_{g}$, and not on $\varepsilon$ nor on $x(\cdot)$. Since $\varepsilon\in(0,1)$ is arbitrary, letting $\varepsilon\downarrow0$ gives
    \begin{eqnarray*}
        d_{\operatorname{Sol}(y_{0})}(x(\cdot)) \leq \|x_{0}-y_{0}\|\exp\left[\int_{0}^{T}\left(L_{F}(s)+\int_{0}^{s}L_{g}(s,\sigma)d\sigma\right)ds\right].
    \end{eqnarray*}
    Since $x(\cdot)\in\operatorname{Sol}(x_{0})$ was arbitrary, taking the supremum over $x(\cdot)\in\operatorname{Sol}(x_{0})$ establishes \eqref{eq:Lemma9-OneSided}. As indicated above, interchanging the roles of $x_{0}$ and $y_{0}$ in this same argument, starting instead from an arbitrary $y(\cdot)\in\operatorname{Sol}(y_{0})$, yields the symmetric bound. Consequently,
    \begin{eqnarray*}
        \operatorname{Haus}(\operatorname{Sol}(x_{0}),\operatorname{Sol}(y_{0})) \leq \|x_{0}-y_{0}\|\exp\left[\int_{0}^{T}\left(L_{F}(s)+\int_{0}^{s}L_{g}(s,\sigma)d\sigma\right)ds\right].
    \end{eqnarray*}
\end{proof}
\begin{proof}[Proof of Theorem \ref{theo:Existence-MainResult}]
    Let $S\colon C([0,T];\mathcal{H})\rightrightarrows C([0,T];\mathcal{H})$ be the multivalued mapping defined by
    \begin{align*}
        S(z(\cdot)):=\operatorname{Sol}(h(z(\cdot)))\qquad\forall z(\cdot)\in C([0,T];\mathcal{H}),
    \end{align*}
    where $\operatorname{Sol}(h(z(\cdot)))$ denotes the set of all solutions $x(\cdot)\in C([0,T];\mathcal{H})$ of
    \begin{align*}
        \begin{cases}
            \displaystyle\dot{x}(t) \in -A(t)x(t) + F(t,x(t)) + \int_{0}^{t} g(t,s,x(s))ds \quad\textrm{a.e. }t\in[0,T],\\
            x(0) = h(z(\cdot)) \in\overline{D(A(0))}.
        \end{cases}
    \end{align*}
    By Lemma \ref{lem:ExistenceResultPart2}, $S(z(\cdot))$ is nonempty for every $z(\cdot)\in C([0,T];\mathcal{H})$. Moreover, from Lemma \ref{lem:ExistenceResultPart3} and $(\mathcal{H}_{h})$, it follows that, for every $z_{1}(\cdot),z_{2}(\cdot)\in C([0,T];\mathcal{H})$,
    \begin{align*}
        \operatorname{Haus}(S(z_{1}(\cdot)),S(z_{2}(\cdot)))
        &\leq \|h(z_{1}(\cdot))-h(z_{2}(\cdot))\|\exp\left[\int_{0}^{T}\left(L_{F}(s)+\int_{0}^{s}L_{g}(s,\sigma)d\sigma\right)ds\right]\\
        &\leq L_{h}\exp\left[\int_{0}^{T}\left(L_{F}(s)+\int_{0}^{s}L_{g}(s,\sigma)d\sigma\right)ds\right]\|z_{1}-z_{2}\|_{\infty}.
    \end{align*}
    Therefore, if
    \begin{eqnarray*}
        L_{h}\exp\left[\int_{0}^{T}\left(L_{F}(s)+\int_{0}^{s}L_{g}(s,\sigma)d\sigma\right)ds\right]<1,
    \end{eqnarray*}
    the multivalued mapping $S\colon C([0,T];\mathcal{H})\rightrightarrows C([0,T];\mathcal{H})$ is a contraction with respect to the Hausdorff distance. Since $S$ has nonempty closed values, Nadler's Fixed Point Theorem (see Proposition \ref{prop:NadlerFixedPoint}) ensures the existence of a fixed point $x(\cdot)\in C([0,T];\mathcal{H})$, that is,
    \begin{align*}
        x(\cdot)\in S(x(\cdot))=\operatorname{Sol}(h(x(\cdot))).
    \end{align*}
    Consequently, $x(\cdot)$ satisfies
    \begin{align*}
        \begin{cases}
            \displaystyle\dot{x}(t) \in -A(t)x(t) + F(t,x(t)) + \int_{0}^{t} g(t,s,x(s))ds \quad\textrm{a.e. }t\in[0,T],\\
            x(0) = h(x(\cdot))\in\overline{D(A(0))},
        \end{cases}
    \end{align*}
    and hence $x(\cdot)$ is a solution to \eqref{eq:DynamicWithNonLocalInitialCondition}. This concludes the proof. 
\end{proof}
\begin{remark}[Iterative scheme from Nadler's Fixed Point Theorem]
    The proof of Theorem \ref{theo:Existence-MainResult} also provides a Picard-type iterative scheme. Indeed, let $z_{0}(\cdot)\in C([0,T];\mathcal{H})$ be arbitrary and choose
    \begin{align*}
        z_{1}(\cdot)\in \operatorname{Sol}(h(z_{0}(\cdot))).
    \end{align*}
    Once $z_{n}(\cdot)$ has been constructed, with $z_{n}(\cdot)\in \operatorname{Sol}(h(z_{n-1}(\cdot)))$, choose
    \begin{align*}
        z_{n+1}(\cdot)\in \operatorname{Sol}(h(z_{n}(\cdot)))
    \end{align*}
    such that
    \begin{align*}
        \|z_{n+1}-z_{n}\|_{\infty}
        \leq \alpha\,\operatorname{Haus}(\operatorname{Sol}(h(z_{n}(\cdot))),\operatorname{Sol}(h(z_{n-1}(\cdot)))),
    \end{align*}
    for some $\alpha>1$ satisfying
    \begin{align*}
        \alpha L_{h}\exp\left[\int_{0}^{T}\left(L_{F}(s)+\int_{0}^{s}L_{g}(s,\sigma)d\sigma\right)ds\right]<1.
    \end{align*}
    Since $z_{n+1}(\cdot)\in \operatorname{Sol}(h(z_{n}(\cdot)))$, it follows that $z_{n+1}(\cdot)$ is a solution of
    \begin{align*}
        \begin{cases}
            \displaystyle\dot{z}_{n+1}(t)\in -A(t)z_{n+1}(t) + F(t,z_{n+1}(t)) + \int_{0}^{t}g(t,s,z_{n+1}(s))ds\quad\textrm{a.e. }t\in[0,T],\\
            z_{n+1}(0)=h(z_{n}(\cdot)).
        \end{cases}
    \end{align*}
    Moreover, by Lemma \ref{lem:ExistenceResultPart3},
    \begin{align*}
        \operatorname{Haus}(\operatorname{Sol}(h(z_{n}(\cdot))),\operatorname{Sol}(h(z_{n-1}(\cdot))))
        \leq L_{h}\exp\left[\int_{0}^{T}\left(L_{F}(s)+\int_{0}^{s}L_{g}(s,\sigma)d\sigma\right)ds\right]\|z_{n}-z_{n-1}\|_{\infty},
    \end{align*}
    and therefore
    \begin{align*}
        \|z_{n+1}-z_{n}\|_{\infty}
        \leq \alpha L_{h}\exp\left[\int_{0}^{T}\left(L_{F}(s)+\int_{0}^{s}L_{g}(s,\sigma)d\sigma\right)ds\right]\|z_{n}-z_{n-1}\|_{\infty}.
    \end{align*}
    Setting
    \begin{align*}
        q:=\alpha L_{h}\exp\left[\int_{0}^{T}\left(L_{F}(s)+\int_{0}^{s}L_{g}(s,\sigma)d\sigma\right)ds\right]<1,
    \end{align*}
    one obtains
    \begin{align*}
        \|z_{n+1}-z_{n}\|_{\infty}\leq q\|z_{n}-z_{n-1}\|_{\infty},
    \end{align*}
    and consequently
    \begin{align*}
        \|z_{n+1}-z_{n}\|_{\infty}\leq q^{n}\|z_{1}-z_{0}\|_{\infty}.
    \end{align*}
    Hence, $(z_{n})_{n}$ is a Cauchy sequence in $C([0,T];\mathcal{H})$, and therefore converges uniformly to some $x(\cdot)\in C([0,T];\mathcal{H})$.

    Finally, since $z_{n+1}(\cdot)\in \operatorname{Sol}(h(z_{n}(\cdot)))$, one has
    \begin{align*}
        d_{\operatorname{Sol}(h(x(\cdot)))}(z_{n+1}(\cdot))
        \leq \operatorname{Haus}(\operatorname{Sol}(h(z_{n}(\cdot))),\operatorname{Sol}(h(x(\cdot)))).
    \end{align*}
    By Lemma \ref{lem:ExistenceResultPart3} and $(\mathcal{H}_{h})$, the right-hand side converges to zero as $n\to+\infty$. Since $z_{n+1}\to x$ uniformly and $\operatorname{Sol}(h(x(\cdot)))$ is closed, it follows that
    \begin{align*}
        x(\cdot)\in \operatorname{Sol}(h(x(\cdot))).
    \end{align*}
    Therefore, $x(\cdot)$ is a solution to Problem \eqref{eq:DynamicWithNonLocalInitialCondition}.
\end{remark}

\section{Selected Examples and Remarks} 

\paragraph{Examples of Autonomous Maximal Monotone Operators:} Some examples of autonomous maximal monotone operators are presented below. In addition, several comments concerning the results obtained in the previous section are provided, and the application of Theorem \ref{theo:Existence-MainResult} to establish the existence of solutions for models related to the heat equation and the sweeping process is discussed. These models are introduced in detail below.

\begin{ejemplo}
    \label{ejemplo:MaximalMonotoneOperators}
    The following examples of maximal monotone operators are independent of $t\in[0,T]$ and illustrate several classes of dynamics covered by Theorem \ref{theo:Existence-MainResult2}.
    \begin{enumerate}
        \item Let $\mathcal{H}=\mathbb{R}^{n}$ and let $M\in\mathcal{M}_{n}(\mathbb{R})$ be a symmetric positive semidefinite matrix. Then $Ax:=Mx$ is maximal monotone with $D(A)=\mathbb{R}^{n}$.
        
        \item\textit{Subdifferential of a convex function \cite{MR1491362}:} Let $\varphi\colon\mathcal{H}\to\mathbb{R}\cup\{+\infty\}$ be a proper, convex, and lower semicontinuous function, that is, $\varphi\in\Gamma_{0}(\mathcal{H})$. Then the operator $A\colon\mathcal{H}\rightrightarrows\mathcal{H}$ defined by $A(x):=\partial\varphi(x)$ is maximal monotone, where:
        \begin{align*}
            \partial\varphi(x):=\{\xi\in\mathcal{H}:\varphi(y)-\varphi(x)\geq\langle\xi,y-x\rangle\quad\forall y\in\mathcal{H}\} \quad\forall x\in\operatorname{dom}(\varphi)
        \end{align*}
        denotes the subdifferential of $\varphi$. In this case,
        \begin{align*}
            D(A)=\{x\in\mathcal{H}:\partial\varphi(x)\neq\emptyset\}\subseteq\operatorname{dom}(\varphi),
        \end{align*}
        where
        \begin{align*}
            \operatorname{dom}(\varphi):=\{x\in\mathcal{H}:\varphi(x)<+\infty\}.
        \end{align*}
        Moreover, by the Brøndsted-Rockafellar Theorem (proved in \cite{MR178103}), one has $\overline{D(A)}=\operatorname{dom}(\varphi)$.
        
        \item\textit{Laplacian with Dirichlet boundary condition:} Let $\mathcal{H}=L^{2}(\Omega)$, where $\Omega\subseteq\mathbb{R}^{n}$ is a sufficiently regular open set. Consider $Au=-\Delta u$, with domain $D(A)=H^{2}(\Omega)\cap H_{0}^{1}(\Omega)$. Then $A$ is maximal monotone in $L^{2}(\Omega)$. Moreover, $A$ is the subdifferential of the functional $\varphi\colon L^{2}(\Omega)\to\mathbb{R}\cup\{+\infty\}$ given by
        \begin{align*}
            \varphi(u):=
            \begin{cases}
                \displaystyle\frac{1}{2}\int_{\Omega}|\nabla u(x)|^{2}dx &\textrm{if }u\in H_{0}^{1}(\Omega),\\
                +\infty&\textrm{otherwise}.
            \end{cases}
        \end{align*}

        \item\textit{$p$-Laplacian with Dirichlet boundary condition:} Let $\mathcal{H}=L^{2}(\Omega)$ as in the previous case. Consider the $p$-Laplacian
        \begin{align*}
            Au=-\operatorname{div}(|\nabla u|^{p-2}\nabla u),
        \end{align*}
        for $p>1$. This operator is the subdifferential of the functional $\varphi\colon L^{2}(\Omega)\to\mathbb{R}\cup\{+\infty\}$ given by
        \begin{align*}
            \varphi(u) :=
            \begin{cases}
                \displaystyle\frac{1}{p}\int_{\Omega}|\nabla u(x)|^{p}dx &\textrm{if }u \in L^{2}(\Omega)\cap W_{0}^{1,p}(\Omega),\\
                +\infty&\textrm{otherwise}.
            \end{cases}
        \end{align*}
        Therefore, $\partial\varphi(u)=Au$, for every $u\in D(A)$, where
        \begin{align*}
            D(A) = \{u\in L^{2}(\Omega)\cap W_{0}^{1,p}(\Omega):  -\operatorname{div}(|\nabla u|^{p-2}\nabla u)\in L^{2}(\Omega) \}.
        \end{align*}

        \item\textit{Normal cone to a convex set:} Let $C\subseteq\mathcal{H}$ be a nonempty, closed and convex set. The normal cone to $C$ is defined by
        \begin{align*}
            N_{C}(x):=\{\xi\in\mathcal{H}:\langle\xi,y-x\rangle\leq0\quad\forall y\in C\}\quad\forall x\in C,
        \end{align*}
        and $N_{C}(x):=\emptyset$ if $x\notin C$. Then $A=N_{C}$ is maximal monotone and $D(A)=C$. Indeed, $N_{C}$ is the subdifferential of the indicator function $\delta_{C}\colon\mathcal{H}\to\mathbb{R}\cup\{+\infty\}$, given by
        \begin{align*}
            \delta_{C}(x)=
            \begin{cases}
                0&\textrm{if }x\in C,\\
                +\infty&\textrm{if }x\notin C.
            \end{cases}
        \end{align*}

        \item\textit{Obstacle operator:} Let $\mathcal{H}=L^{2}(\Omega)$, and let $C\subseteq L^{2}(\Omega)$ be a nonempty, closed and convex set of the form
        \begin{align*}
            C = \{u\in L^{2}(\Omega):u(\omega)\geq \psi(\omega) \quad\textrm{a.e. } \omega\in\Omega\},
        \end{align*}
        for a suitable obstacle function $\psi\colon\Omega\to\mathbb{R}$. Let $\varphi\colon\mathcal{H}\to\mathbb{R}\cup\{+\infty\}$ be a proper, convex and lower semicontinuous function. If the standard qualification condition for the subdifferential sum rule holds (see section 16.4 in \cite{MR3616647}), then:
        \begin{align*}
            A=\partial\varphi+N_{C}
        \end{align*}
        is maximal monotone. Equivalently, in this case,
        \begin{align*}
            A=\partial(\varphi+\delta_{C}).
        \end{align*}
        Such operators appear naturally in variational inequalities with obstacles.
    \end{enumerate}
\end{ejemplo}

For more details on the treatment of the Laplace operator, the reader is referred to Chapter 10 of \cite{MR2759829}. Concerning maximal monotone operators in the context of partial differential equations, the reader is referred to \cite{MR1422252}.

\begin{remark}[A Priori Bounds in Theorem \ref{teo:BoundsForEpsilonSolutionsAndHisSelections}]
    \vphantom{}
    \begin{enumerate}
        \item Theorem \ref{teo:BoundsForEpsilonSolutionsAndHisSelections} provides an a priori bound for $\varepsilon$-solutions and solutions of the Cauchy problem \eqref{eq:DynamicWithLocalInitialCondition}. More precisely, for a fixed initial condition, there exist $M>0$ and $\mu\in L^{1}([0,T];\mathbb{R}_{+})$ such that every $\varepsilon$-solution $x(\cdot)=x(x_{0},f)(\cdot)$ satisfies
        \begin{align*}
            \|x(t)\|\leq M\quad\forall t\in[0,T],
        \end{align*}
        and its associated perturbation
        \begin{align*}
            f_{x}(t)\in F(t,x(t)+\varepsilon\mathbb{B})+\int_{0}^{t}g(t,s,x(s))ds
        \end{align*}
        is uniformly integrably bounded by $\mu(\cdot)$, in the sense that
        \begin{align*}
            \|F(t,x(t))\|_{\max}+\int_{0}^{t}\|g(t,s,x(s))\|ds\leq\mu(t)\quad\textrm{for a.e. }t\in[0,T].
        \end{align*}
        However, these bounds are tied to the fixed Cauchy initial condition. Therefore, if $(x_{0}^{n})_{n}\subseteq\overline{D(A)}$ is a sequence of initial conditions and $(x_{n})_{n}$ is a sequence of $\varepsilon$-solutions satisfying $x_{n}(0)=x_{0}^{n}$, one cannot conclude, in general, the existence of a single constant $M>0$ such that
        \begin{align*}
            \|x_{n}(t)\|\leq M\quad\forall t\in[0,T],\ \forall n\in\mathbb{N}.
        \end{align*}
        For instance, let $C\subseteq\mathbb{R}^{m}$ be a nonempty, closed, convex and unbounded set, and let $A=N_{C}$. Then $A$ is maximal monotone and $D(A)=C$. If $F=0$ and $g=0$, the dynamics reduce to
        \begin{align*}
            \dot{x}(t)\in -N_{C}(x(t))\quad\textrm{a.e. }t\in[0,T].
        \end{align*}
        For every $x_{0}\in C$, the unique solution is the constant trajectory
        \begin{align*}
            x(t)=x_{0}\quad\forall t\in[0,T].
        \end{align*}
        Indeed, $0\in N_{C}(x_{0})$, and therefore $\dot{x}(t)=0\in -N_{C}(x_{0})$ for every $t\in[0,T]$. Hence, if $(x_{0}^{n})_{n}$ in $C$ is such that $\lim_{n\to\infty}\|x_{0}^{n}\|=+\infty$, and $x_{n}(\cdot)$ denotes the corresponding solution, then
        \begin{align*}
            x_{n}(t)=x_{0}^{n}\quad\forall t\in[0,T].
        \end{align*}
        Consequently,
        \begin{align*}
            \lim_{n\to\infty}\|x_{n}\|_{\infty}=\lim_{n\to\infty}\|x_{0}^{n}\|=+\infty,
        \end{align*}
        and no uniform bound can hold for arbitrary unbounded sequences of initial conditions.

        \item Furthermore, Theorem \ref{teo:BoundsForEpsilonSolutionsAndHisSelections} does not provide a uniform estimate for $\dot{x}(\cdot)$. This is due to the fact that the maximal monotone term may be unbounded, even when the perturbation is bounded. For example, in $\mathbb{R}$, consider the proper, convex and lower semicontinuous function:
        \begin{align*}
            \varphi(x):=
            \begin{cases}
                -\log x, & x>0,\\
                +\infty, & x\leq0,
            \end{cases}
        \end{align*}
        and let $A=\partial\varphi$. Then $A$ is maximal monotone and $Ax = \{-1/x\}$ for all $x>0$. The homogeneous inclusion $\dot{x}(t)\in -Ax(t)$, becomes
        \begin{align*}
            \dot{x}(t)=\frac{1}{x(t)} \quad\textrm{a.e. }t\in[0,T].
        \end{align*}
        If $x_{n}(0)=1/n$, then
        \begin{align*}
            x_{n}(t)=\sqrt{\frac{1}{n^{2}}+2t} \quad\forall t\in[0,T],
        \end{align*}
        and therefore
        \begin{align*}
            \dot{x}_{n}(t)=\frac{1}{\sqrt{\frac{1}{n^{2}}+2t}} = \sqrt{\frac{n^{2}}{1+2n^{2}t}} \quad\textrm{a.e. }t\in[0,T].
        \end{align*}
        Although the trajectories $(x_{n})_{n}$ are uniformly bounded on $[0,T]$, one has $\|\dot{x}_{n}\|_{L^{\infty}(0,T)}=n$ for all $n\in\mathbb{N}$, which shows that no uniform bound for the velocities can be expected in general.
    \end{enumerate}
\end{remark}

\paragraph{Parabolic Inclusion Driven by the $p$-Laplacian:} The next example illustrates the dynamic \eqref{eq:DynamicWithNonLocalInitialCondition} when $A$ is chosen as the maximal monotone realization of the Dirichlet $p$-Laplacian, as described in Example \ref{ejemplo:MaximalMonotoneOperators}.4. This choice leads to a nonlinear parabolic inclusion driven by the $p$-Laplacian. Such equations constitute a natural generalization of the heat equation, since the case $p=2$ reduces to the classical Dirichlet Laplacian. For $p\neq2$, the diffusion becomes nonlinear and depends on the magnitude of the spatial gradient, through the flux law $|\nabla u|^{p-2}\nabla u$. Therefore, this class of problems is suitable for modeling nonlinear diffusion phenomena arising, for instance, in non-Newtonian fluids, nonlinear elasticity, filtration through porous media, and variational models associated with energies of $p$-growth. In the present framework, the Volterra term incorporates memory effects, while the nonlocal initial condition allows the initial state to depend on the whole evolution of the system.

\begin{ejemplo}
    \label{ejemplo:pLaplacianDifferentialInclusion}
    Let $Au=-\operatorname{div}(|\nabla u|^{p-2}\nabla u)$, with $p>1$, be the maximal monotone operator described in Example \ref{ejemplo:MaximalMonotoneOperators}.4. Then, the abstract problem becomes the following nonlocal parabolic inclusion:
    \begin{align*}
        \begin{cases}
            \partial_{t}u(t,\xi)-\operatorname{div}(|\nabla_{\xi}u(t,\xi)|^{p-2}\nabla_{\xi}u(t,\xi))\in F(t,u(t))(\xi)+\int_{0}^{t}g(t,s,u(s))(\xi)ds&\quad\textrm{a.e. }(t,\xi)\in[0,T]\times\Omega,\\
            u(t,\xi)=0&\quad\textrm{for }(t,\xi)\in[0,T]\times\partial\Omega,\\
            u(0,\xi)=h(u(\cdot))(\xi)&\quad\textrm{for }\xi\in\Omega.
        \end{cases}
    \end{align*}
    For instance, one may take
    \begin{align*}
        F(t,u):=a(t,\cdot)u+\beta(t)\mathbb{B}_{L^{2}(\Omega)},
    \end{align*}
    where $a\in L^{1}([0,T];L^{\infty}(\Omega))$, $\beta\in L^{1}([0,T];\mathbb{R}_{+})$, and
    \begin{align*}
        g(t,s,u):=K(t,s,\cdot)u,
    \end{align*}
    with $K\in L^{2}(D;L^{\infty}(\Omega))$. Finally, define
    \begin{align*}
        h(u(\cdot)):=u_{*}+\lambda\int_{0}^{T}u(s)ds,
    \end{align*}
    where $u_{*}\in L^{2}(\Omega)$ and $\lambda\geq0$ is chosen so that $h(u(\cdot))\in\overline{D(A)}=L^{2}(\Omega)$. Moreover,
    \begin{align*}
        \|h(u(\cdot))-h(v(\cdot))\|_{L^{2}(\Omega)}
        \leq \lambda\int_{0}^{T}\|u(s)-v(s)\|_{L^{2}(\Omega)}ds
        \leq \lambda T\|u-v\|_{\infty}.
    \end{align*}
    Hence, $L_{h}=\lambda T$. Therefore, by \ref{theo:Existence-MainResult2}, this nonlocal parabolic inclusion admits at least one solution provided that
    \begin{align*}
        \lambda T\exp\left[\int_{0}^{T}\left(\|a(s,\cdot)\|_{L^{\infty}(\Omega)}+\int_{0}^{s}\|K(s,\sigma,\cdot)\|_{L^{\infty}(\Omega)}d\sigma\right)ds\right]<1.
    \end{align*}
\end{ejemplo}

\begin{remark}[Heat equation with multivalued OSL and Volterra perturbations]
    In particular, the previous example includes the classical heat equation as a special case. Indeed, when $p=2$, the $p$-Laplacian operator reduces to the Dirichlet Laplacian, namely $A=-\Delta$. Hence, the abstract nonlocal inclusion gives rise to the following nonlocal parabolic inclusion:
    \begin{align*}
        \begin{cases}
            \partial_{t}u(t,\xi)-\Delta_{\xi}u(t,\xi)\in F(t,u(t))(\xi)+\int_{0}^{t}g(t,s,u(s))(\xi)ds&\quad\textrm{a.e. }(t,\xi)\in[0,T]\times\Omega,\\
            u(t,\xi)=0&\quad\textrm{for }(t,\xi)\in[0,T]\times\partial\Omega,\\
            u(0,\xi)=h(u(\cdot))(\xi)&\quad\textrm{for }\xi\in\Omega.
        \end{cases}
    \end{align*}
    This shows that Theorem \ref{theo:Existence-MainResult2} applies, in particular, to heat-type equations with memory effects and nonlocal initial conditions. The term $\int_{0}^{t}g(t,s,u(s))ds$ represents a Volterra perturbation depending on the past history of the temperature distribution, while the condition $u(0,\xi)=h(u(\cdot))(\xi)$ allows the initial state to depend on the whole evolution of the system.
\end{remark}

\begin{remark}[A first-order antiplane elastic model with nonlinear stress]
    The $p$-Laplacian appears naturally in overdamped antiplane elastic models, see \cite{MR1327716} for a survey of antiplane shear deformations in linear and nonlinear solid mechanics, on which the reduction below is based. Let $\Omega\subseteq\mathbb{R}^{2}$ be a bounded domain and consider an antiplane displacement field of the form
    \begin{align*}
        \mathbf{u}(t,\xi)=(0,0,w(t,\xi)), \qquad \xi:=(\xi_{1},\xi_{2})\in\Omega.
    \end{align*}
    This ansatz is a classical reduction in solid mechanics, since it reduces a three-dimensional deformation to a scalar displacement depending on two spatial variables.
    Let $\boldsymbol{\sigma}$ denote the Cauchy stress tensor. In the antiplane setting, the relevant components of $\boldsymbol{\sigma}$ are the shear stresses $\sigma_{31}$ and $\sigma_{32}$. Assume that these components satisfy the nonlinear constitutive law
    \begin{align*}
        \sigma_{31} = |\nabla_{\xi}w|^{p-2}\partial_{\xi_{1}}w
        \quad\textrm{and}\quad
        \sigma_{32} = |\nabla_{\xi}w|^{p-2}\partial_{\xi_{2}}w.
    \end{align*}
    Thus, the third component of the internal force $\operatorname{div}\boldsymbol{\sigma}$ is given by
    \begin{align*}
        (\operatorname{div}\boldsymbol{\sigma})_{3}
        = \partial_{\xi_{1}}\sigma_{31} + \partial_{\xi_{2}}\sigma_{32}
        = \operatorname{div}_{\xi}\left(|\nabla_{\xi}w|^{p-2}\nabla_{\xi}w\right).
    \end{align*}
    Hence, the $p$-Laplacian appears as the divergence of the nonlinear shear stress law.
    \begin{enumerate}
        \item In a fully dynamic model, the balance of linear momentum contains the inertial term $\rho\partial_{tt}w$, where $\rho>0$ denotes the mass density. In contrast, in an overdamped regime, inertial effects are neglected when compared with frictional drag and elastic restoring forces. The balance law is then written in first order in time as
        \begin{align*}
            \eta\partial_{t}w(t,\xi) - (\operatorname{div}\boldsymbol{\sigma})_{3}(t,\xi)
            = b(t,\xi) + \int_{0}^{t}g(t,s,w(s))(\xi)ds,
        \end{align*}
        where $\eta>0$ is an external frictional drag coefficient, $b$ is an external body force, and the Volterra term represents hereditary effects. Therefore, after taking $\eta=1$, one obtains
        \begin{align*}
            \partial_{t}w(t,\xi) - \operatorname{div}_{\xi} \left( |\nabla_{\xi}w(t,\xi)|^{p-2}\nabla_{\xi}w(t,\xi)\right)
            = b(t,\xi) + \int_{0}^{t} g(t,s,w(s))(\xi)ds.
        \end{align*}
        \item The first-order character of the model is a consequence of the overdamped regime. More precisely, instead of the inertial term $\rho\partial_{tt}w$, the dominant dissipative mechanism is represented by the external frictional drag term $\eta\partial_{t}w$. Thus, the resulting evolution is parabolic rather than hyperbolic. The Volterra term accounts for hereditary effects, as in viscoelastic materials with fading memory, see for instance \cite{MR1153021}. For related variational models in viscoelastic contact mechanics, see also \cite{SofoneaHanShillor2006}.
    \end{enumerate}
\end{remark}

\paragraph{Examples of Time-Dependent Maximal Monotone Operators:} The following example illustrates several classes of time-dependent maximal monotone operators satisfying Vladimirov's absolute continuity condition. In particular, it shows that this assumption naturally arises for subdifferential operators associated with convex functions whose time dependence is introduced through translations, affine perturbations, or quadratic terms. These constructions provide a useful class of nonautonomous evolution problems covered by the general framework considered in this work and show that Vladimirov's condition is compatible with several standard time-dependent convex energies.

\begin{ejemplo}[Vladimirov absolute continuity]\label{vladac} 
    Let $\mathcal{H}$  be a real Hilbert space and let $\Gamma_{0}(\mathcal{H})$ be the set of extended real-valued, proper, lower semicontinuous and convex functions.
    \begin{enumerate}
        \item \textit{Translations of an arbitrary convex function:} Let $g(\cdot)\in \Gamma_{0}(\mathcal{H})$ and $c(\cdot)\in W^{1,1}([0,1];\mathcal{H})$ and define the function $f\in \Gamma_0(\mathcal{H})$ by 
        \begin{align*}
            f(t,x)=g(x-c(t)).
        \end{align*}
        Then 
        \begin{align*}
            \partial f(t,x)=\partial g(x-c(t)), \quad \forall x\in \mathcal{H}, \ \forall t\in [0, 1].
        \end{align*}
        and:
        \begin{align*}
            \operatorname{dis}(\partial f(t_1, \cdot), \partial f(t_2, \cdot)) \leq \Vert c(t_1)-c(t_2)\Vert.
        \end{align*}
        If $a(t)=\int_{0}^{t}\Vert c(\tau)\Vert d\tau$ then $a\in W^{1,1}([0,1];\mathbb{R})$ and $\Vert c(t_{1})-c(t_{2})\Vert\leq \vert a(t_{1})-a(t_{2})\vert$. 
        Thus the required Vladimirov absolute continuity condition holds.
    
        \item \textit{Adding a time-dependent affine term:} The previous class can be enlarged by considering
        $$f(t,x)=g(x-c(t))+⟨p(t),x⟩+r(t),$$
        where  $g\in \Gamma_0(\mathcal{H})$ and $c, p \in W^{1,1}((0,1), \mathcal{H})$ and $r\in W^{1,1}(0, 1)$. 
        Note that  the scalar term $r(t)$  does not affect the subdifferential, $\partial f(t,x)=\partial g(x-c(t))+p(t).$ A convenient sufficient condition for Vladimirov absolute continuity is therefore $c, p \in W^{1,1}((0,1), \mathcal{H})$ together with suitable bounds on the domains of the subdifferentials. In particular, when $g$ has bounded subgradients on its effective domain, the estimate becomes
        $$dis(\partial f(t_1, \cdot), \partial f(t_2, \cdot)) \leq C(\Vert c(t_1)-c(t_2)\Vert+ \Vert p(t_1)-p(t_2)\Vert )$$
        for an appropriate constant $C>0$. Therefore, it is enough to set $a(t) = C\int_0^t (\Vert \dot{c}(s)\Vert + \Vert \dot{p}(s)\Vert)ds$, in order to get
        $$ dis(\partial f(t_1, \cdot), \partial f(t_2, \cdot)) \leq \vert a(t_1)-a(t_2)\vert.$$
        
        \item \textit{Time-dependent quadratic perturbations:} Let
        $$f(t,x)=g(x-c(t))+{1\over 2}\langle Q(t)x,x\rangle+\langle p(t),x\rangle+r(t),$$
        where  $g\in \Gamma_0(\mathcal{H})$, $c, p \in W^{1,1}((0,1), \mathcal{H})$, $r\in W^{1,1}(0, 1)$ and $Q(t): \mathcal{H}\mapsto \mathcal{H}$ is a linear continuous operator which is self-adjoint and positive semidefinite for every $t$. Assume $Q\in W^{1,1}((0,1), \mathcal{L}(\mathcal{H}))$, where 
        $\mathcal{L}(\mathcal{H})$ denotes the linear space of bounded operators.  Then $f(t, \cdot) \in \Gamma_0(\mathcal{H})$ and 
        $$\partial f(t,x)=\partial g(x-c(t))+Q(t)x+p(t).$$
        If the family $Q(t)$  is uniformly bounded then the family of subdifferentials has absolutely continuous variation in the
        Vladimirov sense under the usual boundedness assumptions on the relevant graphs.
    \end{enumerate}
\end{ejemplo}

\begin{remark}[A finite-dimensional parameterization] 
    A very useful general construction is
    \begin{align*}
        f(t,x) := g(x) + \sum_{k=1}^{m} \alpha_{k}(t)g_{k}(x) + \langle p(t),x\rangle + r(t),
    \end{align*}
    where $g$, $g_{k}$, $\alpha_{k}$, $p(t)$ and $r(t)$ satisfies appropriate assumptions to ensure the convexity and the Vladimirov absolute continuity. 
\end{remark}

\paragraph{Volterra Sweeping Process:} The dynamic \eqref{eq:DynamicWithNonLocalInitialCondition} is illustrated in the next example in the particular case where $A(t)=N_{C(t)}$, with $(C(t))_{t\in[0,T]}$ being a family of nonempty, closed, and convex subsets of $\mathbb{R}^{n}$. In this setting, the dynamics can be interpreted as a perturbed sweeping process governed by a moving convex constraint, together with a Volterra memory term and a nonlocal initial condition. Differential inclusions involving normal cones to moving convex sets were introduced by J.-J. Moreau in a series of seminal papers \cite{MR0637727, MR0637728, MR0508661}, and have since been used as a natural modeling framework for constrained evolution problems, with applications in contact mechanics, electrical circuits, crowd motion, and related areas. A history-dependent variant, known as the Volterra sweeping process, was introduced later in \cite{MR4099068} by incorporating an integral term into the classical perturbed sweeping process. For well-posedness results concerning Volterra sweeping processes, the reader is referred to \cite{MR4836331, MR4883326}. The following example combines these features by considering a moving convex set, a nonlinear perturbation satisfying a one-sided Lipschitz condition, an exponentially fading memory term, and a nonlocal initial condition.

\begin{ejemplo}[Moving convex sweeping process with memory]
    \label{example:SweepingProcessWithMemory}
    Let $f\colon\mathcal{H}\to\mathcal{H}$, $g\colon D\times\mathcal{H}\to\mathcal{H}$, and $h\colon C([0,T];\mathcal{H})\to\mathcal{H}$ be the mappings defined by
    \begin{align*}
        &f(x):=
        \begin{cases}
            -\dfrac{x}{\sqrt{\|x\|}}, & x\neq 0,\\
            0, & x=0,
        \end{cases}\\
        &g(t,s,x):=\kappa e^{-\gamma(t-s)}Mx\qquad\forall (t,s)\in D,\ \forall x\in\mathcal{H},\\
        & C(t) := c(t)+D\qquad\forall t\in [0,T],\\
        &h(x(\cdot)):=\rho P_{C(0)}(x(T))\qquad\forall x(\cdot)\in C([0,T];\mathcal{H}),
    \end{align*}
    where $\gamma,\kappa>0$, $0<\rho\leq1$, and $M\colon \mathcal{H}\mapsto\mathcal{H}$ is a continuous operator, with  $\|M\|\leq1$, $c(\cdot)\in W^{1,1}([0,T]; \mathcal{H})$ and $D\subset \mathcal{H}$ is a closed convex set  containing $-c(0)$. Thus, by Example \ref{vladac}, for every $t,s\in[0,T]$,
    \begin{align*}
        \operatorname{Haus}(C(t),C(s))=\|z(t)-z(s)\|\leq\int_{\min\{s,t\}}^{\max\{s,t\}}\|\dot{c}(\tau)\|d\tau.
    \end{align*}
    Hence, by defining
    \begin{align*}
        a(t):=\int_{0}^{t}\|\dot{c}(\tau)\|d\tau,
    \end{align*}
    one has $a\in W^{1,1}([0,T])$ and
    \begin{align*}
        \operatorname{Haus}(C(t),C(s))\leq |a(t)-a(s)|\qquad\forall t,s\in[0,T].
    \end{align*}
    Moreover, the family of maximal monotone operators $A(t):=N_{C(t)}$ satisfies hypothesis $(\mathcal{H}_A$. Furthermore, since
    The mapping $f$ is continuous on $\mathcal{H}$. Indeed, it is continuous on $\mathcal{H}\setminus\{0\}$ and
    \begin{align*}
        \lim_{x\to0}\|f(x)-f(0)\|=\lim_{x\to0}\sqrt{\|x\|}=0.
    \end{align*}
    Moreover, $f$ is not Lipschitz continuous in any neighborhood of the origin, since
    \begin{align*}
        \lim_{x\to0}\frac{\|f(x)-f(0)\|}{\|x\|}=\lim_{x\to0}\frac{1}{\sqrt{\|x\|}}=+\infty.
    \end{align*}
    On the other hand, $f$ satisfies the one-sided Lipschitz condition with constant $0$. Indeed, the mapping $x\mapsto \frac{x}{\sqrt{\|x\|}}$ is the gradient of the convex function
    \begin{align*}
        \varphi(x):=\frac{2}{3}\|x\|^{3/2}.
    \end{align*}
    Hence, this mapping is monotone, and therefore
    \begin{align*}
        \langle f(x)-f(y),x-y\rangle\leq0\qquad\forall x,y\in\mathcal{H}.
    \end{align*}
    In addition,
    \begin{align*}
        \|f(x)\|=\sqrt{\|x\|}\leq1+\|x\|\qquad\forall x\in\mathcal{H},
    \end{align*}
    so that $f$ satisfies a linear growth condition. The corresponding moving sweeping process with memory and nonlocal initial condition is given by
    \begin{align*}
        \begin{cases}
            \displaystyle\dot{x}(t)\in -N_{C(t)}(x(t))+f(x(t))+\int_{0}^{t}g(t,s,x(s))ds\quad\textrm{a.e. }t\in[0,T],\\
            x(0)=h(x(\cdot)).
        \end{cases}
    \end{align*}
    Since $P_{C(0)}(x(T))\in C(0)$ for every $x(\cdot)\in C([0,T];\mathcal{H})$, and since $0<\rho\leq1$ and $-c(0)\in D$, one has
    \begin{align*}
        h(x(\cdot))=\rho P_{C(0)}(x(T))\in C(0).
    \end{align*}
    Moreover, since the metric projection onto a nonempty closed convex set is nonexpansive,
    \begin{align*}
        \|h(x(\cdot))-h(y(\cdot))\|\leq\rho\|x(T)-y(T)\|\leq\rho\|x(\cdot)-y(\cdot)\|_{\infty}.
    \end{align*}
    Therefore, $h$ is $\rho$-Lipschitz. The Volterra perturbation satisfies
    \begin{align*}
        \|g(t,s,x)-g(t,s,y)\|\leq\kappa e^{-\gamma(t-s)}\|x-y\|\qquad\forall (t,s)\in D,\ \forall x,y\in\mathcal{H}.
    \end{align*}
    Therefore,
    \begin{align*}
        \int_{0}^{T}\int_{0}^{s}\kappa e^{-\gamma(s-\tau)}d\tau ds=\frac{\kappa}{\gamma}\left(T-\frac{1-e^{-\gamma T}}{\gamma}\right).
    \end{align*}
    Hence, Corollary \ref{cor:VolterraSweepingProcess} ensures the existence of at least one solution provided that
    \begin{align*}
        \rho\exp\left[\frac{\kappa}{\gamma}\left(T-\frac{1-e^{-\gamma T}}{\gamma}\right)\right]<1.
    \end{align*}
    The parameter $\gamma$ governs the exponential fading of memory, so that larger values of $\gamma$ correspond to faster forgetting of past states. The parameter $\kappa$ controls the magnitude of the memory contribution, $\beta$ determines the amplitude of the motion of the constraint set, and $\rho$ measures the influence of the terminal state in the nonlocal initial condition. For instance, the choice $M=-I$ produces a dissipative memory term, whereas, in dimension $2$, a rotation matrix $M$ generates a memory contribution with rotational effects.
\end{ejemplo}

\begin{remark}[Iterative scheme for Example \ref{example:SweepingProcessWithMemory}]
    Let $\pi_{N}:=\{t_{k}\in[0,T]:t_{k}=\tfrac{kT}{N}\textrm{ for }k=0,\dots,N\}$ be a uniform partition, and let $\Delta t:=\tfrac{T}{N}$ denote the time step. A semi-implicit Euler discretization of the sweeping process dynamics is given by
    \begin{align*}
        -\frac{x_{n+1}-x_{n}}{\Delta t}+f(x_{n})+\sum_{i=0}^{n}\Delta t\,g(t_{n},t_{i},x_{i})\in N_{C(t_{n+1})}(x_{n+1})\quad\forall n=0,\dots,N-1.
    \end{align*}
    Since $P_{C(t)}=(I+\Delta t N_{C(t)})^{-1}$ for every $t\in[0,T]$, it follows that
    \begin{align*}
        x_{n+1}=P_{C(t_{n+1})}\left[x_{n}+\Delta t\left(f(x_{n})+\sum_{i=0}^{n}\Delta t\,g(t_{n},t_{i},x_{i})\right)\right]\quad\forall n=0,\dots,N-1.
    \end{align*}
    Applying this scheme to Example \ref{example:SweepingProcessWithMemory}, with $D$ the closed unit ball $\mathbb{B}$ of $\mathcal{H}$ and $g(t,s,x)=\kappa e^{-\gamma(t-s)}Mx$, one obtains
    \begin{align*}
        x_{n+1}=P_{C(t_{n+1})}\left[x_{n}+\Delta t\left(f(x_{n})+\kappa\sum_{i=0}^{n}\Delta t\,e^{-\gamma(t_{n}-t_{i})}Mx_{i}\right)\right].
    \end{align*}
    Since $C(t)=c(t)+\mathbb{B}$, the metric projection onto $C(t)$ is explicitly given by
    \begin{align*}
        P_{C(t)}(y)=c(t)+P_{\mathbb{B}}(y-c(t)),
    \end{align*}
    where
    \begin{align*}
        P_{\mathbb{B}}(y)=
        \begin{cases}
            y, & \|y\|\leq1,\\
            \dfrac{y}{\|y\|}, & \|y\|>1.
        \end{cases}
    \end{align*}
    With this construction, an iterative procedure for the nonlocal initial condition is given by:
    \begin{center}
        \begin{minipage}{0.9\textwidth}
        \hrule
        \vspace{0.2cm}
        \noindent\textbf{Scheme.}
        Let $(x_{n}^{0})_{n=0}^{N}$ be an initial discrete trajectory such that $x_{n}^{0}\in C(t_{n})$ for every $n=0,\dots,N$, and let $\varepsilon_{\operatorname{tol}}>0$ be fixed. For every $k\in\mathbb{N}$, define
        \begin{align*}
            x_{0}^{k+1}:=\rho P_{C(0)}(x_{N}^{k}).
        \end{align*}
        Then, for each $n=0,\dots,N-1$, construct recursively
        \begin{align*}
            x_{n+1}^{k+1}:=P_{C(t_{n+1})}\left[x_{n}^{k+1}+\Delta t\left(f(x_{n}^{k+1})+\kappa\sum_{i=0}^{n}\Delta t\,e^{-\gamma(t_{n}-t_{i})}Mx_{i}^{k+1}\right)\right].
        \end{align*}
        If
        \begin{align*}
            \max_{0\leq n\leq N}\|x_{n}^{k+1}-x_{n}^{k}\|<\varepsilon_{\operatorname{tol}},
        \end{align*}
        then the procedure is terminated. Otherwise, replace $k$ by $k+1$ and repeat the construction.
        \vspace{0.2cm}
        \hrule
        \end{minipage}
    \end{center}
    Since each time step is defined through the projection onto $C(t_{n+1})$, the discrete trajectory satisfies
    \begin{align*}
        x_{n}\in C(t_{n})\qquad\forall n=0,\dots,N.
    \end{align*}
    Moreover, since $C(0)=\mathbb{B}$, the nonlocal iteration takes the particular form $x_{0}^{k+1}=\rho P_{\mathbb{B}}(x_{N}^{k})$.
\end{remark}

\section{Conclusions}

We have established an existence framework for differential inclusions involving a nonautonomous maximal monotone operator, an OSL set-valued perturbation, and a Volterra memory term under nonlocal condition conditions. The combination of these three features allows the model to describe nonlinear evolution systems in which the governing operator varies with time, the instantaneous perturbation is multivalued and only one-sided Lipschitz, and the dynamics depend on the past through a memory kernel.

The assumptions imposed on the family $A(t)$ are particularly important. In addition to the maximal monotonicity of $A(t)$, the Vladimirov-type absolute continuity of the operator with respect to time, together with the linear growth condition of the mapping $x\mapsto A^{0}(t)x$, where $A^{0}(t)x$ is the element of minimal norm in $A(t)x$, provides sufficient control of the temporal variation of the graph of $A(t)$. These assumptions compensate for the lack of autonomy and make it possible to construct suitable approximations while retaining uniform estimates for the corresponding solutions.

On the perturbation side, the OSL condition imposed on the set-valued mapping $F(t,x)$, together with suitable measurability assumptions, is sufficient to derive the dissipative estimate required in the analysis, without imposing the classical Lipschitz continuity or noncompactness conditions. This is particularly relevant since the OSL framework allows for genuinely nonsmooth and multivalued nonlinearities that need not be globally Lipschitz or compact. Nevertheless, the one-sided structure provided by the OSL condition is strong enough to control the distance between approximate solutions and, consequently, to derive the stability estimates required for the convergence analysis.

Finally, the measurability and Lipschitz regularity assumptions on the Volterra kernel guarantee that the memory contribution is well defined and can be controlled in the appropriate integral estimates. In combination with the preceding assumptions, the Volterra term can therefore be incorporated into the existence argument without destroying the monotonicity and dissipativity structure of the problem.

Consequently, under these hypotheses, the associated differential inclusion admits a solution in the appropriate evolution space, satisfying the inclusion almost everywhere and incorporating both the instantaneous multivalued dynamics and the hereditary Volterra effect. The result shows that maximal monotonicity, Vladimirov absolute continuity, and linear growth of $A(t)$, together with the OSL property of $F$ and the regularity of the Volterra kernel, form a sufficiently robust set of assumptions for the existence theory of nonautonomous differential inclusions with memory and a Lipschitz local condition.

An important feature of this framework is that it weakens the classical assumptions in two directions simultaneously: The operator is allowed to be nonautonomous, while the perturbation $F$ need only satisfy an OSL condition rather than a two-sided Lipschitz condition and the noncompactness condition. The Volterra formulation further extends the theory to systems with memory. This provides a flexible setting for future developments concerning uniqueness and continuous dependence, long-time behavior, stability, optimal control, and numerical approximation of nonlinear evolution inclusions with hereditary effects.

\appendix
\section{Appendix: Proof of Theorem \ref{Kunze1}}

The following lemma is used in the proof of Theorem \ref{Kunze1}.  Its proof is exactly like that of Lemma 3.3.1 in \cite{MR616449}, modulo a change in the definition of the set ${\cal V}$ in the proof of the lemma.

\begin{lema}[\cite{MR616449}]
    \label{lemma:Partition} 
    Let $c(\cdot) \in L^{1}((0,T);[0,+\infty[)$, and assume, without loss of generality, that $c(\cdot)$ is defined at $0$. Then there exists a subset $C\subset [0,T]$ of null measure such that
    \begin{equation*}
        c(t) \in [0,+\infty[, \qquad \forall t\in [0,T]\backslash C,
    \end{equation*}
    and there exists a sequence of partitions $(\pi_{n})_{n}$ of $[0,T]$, where each partition $\pi_{n}$ is given by
    \begin{align*}
        0=t^{n}_{0}<t^{n}_{1}<\cdots <t_{N_{n}}^{n}=T,
    \end{align*}
    such that:
    \begin{enumerate}
        \item $t^{n}_{i}\notin C$, for each $i=0,\dots, N_{n}-1$. 
        \item $\displaystyle \lim_{n\to +\infty} \max_{1\leq i\leq N_{n}}|t^{n}_{i}-t^{n}_{i-1}| = 0$. 
        \item The functions $c_{n}\colon[0,T]\mapsto [0,+\infty[$ defined by :
        \begin{align*}
            c_{n}(t) = c(t_{i-1})\quad\textrm{on $(t_{i-1}, t_i)$ for each $i=1, \cdots, N_{n}$.}  
        \end{align*}
        satisfies
        \begin{align*}
            \lim_{n\to +\infty} \int_{0}^{T} | c(\tau)-c_n(\tau)| d\tau =0.
        \end{align*}
    \end{enumerate}
\end{lema}

\begin{proof}[Proof of Theorem \ref{Kunze1}]
    The proof is based on a time discretization of the family of maximal monotone operators and is divided into several steps.
    
    \noindent$\circ$ \textit{Step 1 - Construction of an approximate solution:} By Lemma \ref{lemma:Partition}, there exists a subset $C\subset [0,T]$ of measure zero such that
    \begin{equation*}
        c(t) \in [0,+\infty[, \quad \forall t\in [0,T]\backslash C,
    \end{equation*}
    and there exists a sequence of partitions satisfying items $1$--$3$ of the Lemma \ref{lemma:Partition}. For every $i=1,\ldots,N_{n}$, define
    \begin{align*}
        \alpha_{i}^{n} := \int_{t_{i-1}^{n}}^{t_{i}^{n}} |\dot{a}(\tau)|d\tau.
    \end{align*}
    Since $a(\cdot)\in W^{1,1}([0,T])$, one has
    \begin{equation}
        \label{eq:alpha-sum}
        \sum_{i=1}^{N_{n}}\alpha_{i}^{n} = \int_{0}^{T} |\dot{a}(\tau)|d\tau.
    \end{equation}
    For $t\in[t_{i}^{n},t_{i+1}^{n})$, the piecewise constant approximation is defined by $A_{n}(t):=A(t_{i}^{n})$. The hypothesis $(\mathcal{H}_{A}^{1})$ gives
    \begin{equation*}
        \label{eq:dis-alpha}
        \operatorname{dis}(A(t_{i-1}^{n}),A(t_{i}^{n}))\leq \alpha_{i}^{n}\qquad\forall i=1,\ldots,N_{n}.
    \end{equation*}
    A standard consequence of the definition of Vladimirov's pseudo-distance (see \cite{MR1092799}) is that, for every $i=1,\ldots,N_{n}$ and every $x\in D(A(t_{i-1}^{n}))$, one has
    \begin{equation*}
        \label{eq:domain-distance-general}
        d(x,D(A(t_{i}^{n}))) \leq \operatorname{dis}(A(t_{i-1}^{n}),A(t_{i}^{n})).
    \end{equation*}
    Hence (see \cite{MR1092799}),
    \begin{equation}
        \label{eq:domain-distance}
        d(x,D(A(t_{i}^{n}))) \leq \alpha_{i}^{n}.
    \end{equation}
    An approximate solution $u_{n}$ is constructed recursively. At the first step, $u_{n}(0)=u_{0}$ is set. Since $A(0)$ is a maximal monotone operator, the classical theory of evolution equations governed by maximal monotone operators (see \cite{MR348562}) ensures that the differential inclusion
    \begin{align*}
        \begin{cases}
            \dot{u}(t) \in -A(0)u(t)\quad\textrm{a.e. }t\in[0,t_{1}^{n}], \\
            u(0) = u_{0}
        \end{cases}
    \end{align*}
    admits a unique solution on $[0,t_{1}^{n})$. This solution is denoted by $u_{n}$ on this interval. The corresponding left-hand limit at $t_{1}^{n}$ is denoted by
    \begin{align*}
        u^{-}_{n}(t_{1}^{n}):= \lim_{t\nearrow t_{1}^{n}} u_{n}(t).
    \end{align*}
    This limit may fail to belong to the new domain $D(A(t_{1}^{n}))$. Therefore, in order to restart the evolution with the operator $A(t_{1}^{n})$, a point $u_{1}^{n}\in D(A(t_{1}^{n}))$ sufficiently close to $u^{-}_{n}(t_{1}^{n})$ is selected. By the assumption,
    \begin{align*}
        \operatorname{dis}(A(t_{0}^{n}),A(t_{1}^{n})) \leq |a(t_{0}^{n})-a(t_{1}^{n})| \leq \alpha_{1}^{n},
    \end{align*}
    and by the corresponding estimate for the distance between the domains, a point $u_{1}^{n}\in D(A(t_{1}^{n}))$ can be selected such that
    \begin{align*}
        \Vert u_{1}^{n} - u^{-}_{n}(t_{1}^{n})\|\leq 2\alpha_{1}^{n}.
    \end{align*}
    As before, the differential inclusion
    \begin{align*}
        \begin{cases}
            \dot{u}(t) \in -A(t_{1}^{n})u(t)\qquad\textrm{a.e. }t\in[t_{1}^{n},t_{2}^{n}),\\ 
            u(t_{1}^{n}) = u_{1}^{n},
        \end{cases}
    \end{align*}
    admits a unique solution on $[t_{1}^{n},t_{2}^{n})$. This solution is again denoted by $u_{n}$ on this interval.

    Suppose now that $u_{n}$ has already been constructed on $[t_{i-1}^{n},t_{i}^{n})$. At the endpoint $t_{i}^{n}$, the left-hand limit $u_{n}^{-}(t_{i}^{n})$ is first considered, namely,
    \begin{align*}
        u^{-}_{n}(t_{i}^{n}):= \lim_{t\nearrow t_{i}^{n}} u_{n}(t).
    \end{align*}
    A point $u_{i}^{n}\in D(A(t_{i}^{n}))$ sufficiently close to $u_{n}^{-}(t_{i}^{n})$ is then selected so that
    \begin{align*}
        \|u_{i}^{n} - u^{-}_{n}(t_{i}^{n})\|\leq 2\alpha_{i}^{n}.
    \end{align*}
    On every interval $[t_{i}^{n},t_{i+1}^{n})$, $u_{n}$ is defined as the unique solution of
    \begin{equation*}
        \label{eq:piecewise-evolution}
        \begin{cases}
            \dot{u}_{n}(t) \in -A(t_{i}^{n})u_{n}(t) \qquad\textrm{a.e. }t\in [t_{i}^{n},t_{i+1}^{n}),\\
            u_{n}(t_{i}^{n}) = u_{i}^{n}. 
        \end{cases}
    \end{equation*} 
    Since $A(t_{i}^{n})$ is maximal monotone, the classical theory of evolution equations governed by maximal monotone operators ensures the existence and uniqueness of $u_{n}$ on each interval (see \cite{MR348562}).

    \noindent $\circ$ \textit{Step 2 - Uniform bound on the approximate solutions:} Let $A^{0}(t)x$ denote the element of minimal norm in $A(t)x$. From $(\mathcal{H}_{A}^{2})$, one has
    \begin{equation*}
        \label{eq:minimal-section}
        \|A^{0}(t)x\| \leq c(t)(1+\|x\|). 
    \end{equation*}
    On each interval $[t_{i}^{n},t_{i+1}^{n})$, the mapping $u_{n}$ has a right derivative $\tfrac{d^{+}u_{n}}{dt}(t)$ (see Theorem 3.1 in \cite{MR348562}) satisfying
    \begin{align*}
        -\frac{d^{+}u_{n}}{dt}(t) = A^{0}(t_{i}^{n})u_{n}(t).      
    \end{align*}
    Therefore, by the hypothesis $(\mathcal{H}_{A}^{2})$,
    \begin{equation*}
        \label{eq:derivative-estimate}
        \|\dot{u}_{n}(t)\| \leq c(t_{i}^{n})(1+\|u_{n}(t)\|) \qquad\textrm{a.e. } t\in [t_{i}^{n},t_{i+1}^{n}).  
    \end{equation*}
    Then
    \begin{equation}
        \label{eq:derivative-cn}
        \|\dot{u}_{n}(t)\| \leq c_{n}(t)(1+\|u_{n}(t)\|)  \qquad\textrm{a.e. } t\in [t_{i}^{n},t_{i+1}^{n}),
    \end{equation}
    and hence
    \begin{equation*}
        \label{eq1KMM}
        \|u_{n}(t)-u_{n}(t_{i}^{n})\| \leq \int_{t_{i}^{n}}^{t} c_{n}(\tau)(1+\|u_{n}(\tau)\|)d\tau  \qquad \forall t\in[t_{i}^{n},t_{i+1}^{n}).
    \end{equation*}
    By passing to the limit as $t\to(t_{i+1}^{n})^{-}$, it follows that
    \begin{equation*}
        \label{eq2KMM}
        \|u_{n}^{-}(t_{i+1}^{n})-u_{n}(t_{i}^{n})\| \leq \int_{t_{i}^{n}}^{t_{i+1}^{n}} c_{n}(\tau)(1+\|u_{n}(\tau)\|)d\tau.
    \end{equation*}
    Let $t\in[0,T]$. Then there exists $i$ such that $t\in[t_{i}^{n},t_{i+1}^{n}]$. Recalling that $u_{n}(t_{k}^{n})=u_{k}^{n}$ and setting $u_{0}^{n}:=u_{0}$, the increment from $u_{0}$ to $u_{n}(t)$ can be decomposed into the continuous evolutions on the subintervals and the jumps at the partition points. Therefore,
    \begin{align*}
        \|u_{n}(t)-u_{n}(0)\|
        &\leq \|u_{n}(t)-u_{n}(t_{i}^{n})\|
        +\sum_{k=0}^{i-1}\|u_{n}^{-}(t_{k+1}^{n})-u_{n}(t_{k}^{n})\|
        +\sum_{k=1}^{i}\|u_{n}(t_{k}^{n})-u_{n}^{-}(t_{k}^{n})\|\\
        &\leq \int_{t_{i}^{n}}^{t}c_{n}(\tau)(1+\|u_{n}(\tau)\|)d\tau
        +\sum_{k=0}^{i-1}\int_{t_{k}^{n}}^{t_{k+1}^{n}}c_{n}(\tau)(1+\|u_{n}(\tau)\|)d\tau
        +2\sum_{k=1}^{i}\alpha_{k}^{n}\\
        &\leq 2\sum_{k=1}^{i}\alpha_{k}^{n}
        +\int_{0}^{t}c_{n}(\tau)(1+\|u_{n}(\tau)\|)d\tau.
    \end{align*}
    Hence,
    \begin{equation*}
        \label{eq:pre-gronwall}
        \|u_{n}(t)\| \leq \|u_{0}\| + 2\sum_{\substack{i=1\\t_{i}^{n}\leq t}}^{N_{n}}\alpha_{i}^{n} +\int_{0}^{t} c_{n}(\tau)(1+\|u_{n}(\tau)\|) d\tau, \quad \forall t\in [0,T]. 
    \end{equation*}
    By \eqref{eq:alpha-sum},
    \begin{equation*}
        \label{eq:alpha-total}
        \sum_{\substack{i=1\\t_{i}^{n}\leq t}}^{N_{n}}\alpha_{i}^{n} \leq \|\dot{a}\|_{L^{1}([0,T])}. 
    \end{equation*}
    Consequently,
    \begin{equation*}
        \label{eq:gronwall-form}
        \|u_{n}(t)\| \leq \|u_{0}\| + 2\|\dot{a}\|_{L^{1}([0,T])} + \int_{0}^{t} c_{n}(\tau)d\tau + \int_{0}^{t} c_{n}(\tau)\|u_{n}(\tau)\| d\tau.
    \end{equation*}
    Since
    \begin{align*}
        \sup_{n\in\mathbb{N}}\|c_{n}\|_{L^{1}([0,T])}\leq \|c\|_{L^{1}([0,T])}+T,
    \end{align*}
    Gronwall's inequality yields
    \begin{equation*}
        \label{eq:gronwall-result}
        \|u_{n}(t)\| \leq \big(\|u_{0}\| + 2\|\dot{a}\|_{L^{1}([0,T])} + \|c\|_{L^{1}([0,T])} + T\big)\exp\big[\|c\|_{L^{1}([0,T])} + T \big].
    \end{equation*}
    Therefore, there exists a constant $M>0$, independent of $n$, such that
    \begin{equation}
        \label{eq:uniform-u-bound}
        \sup_{n\in\mathbb{N}} \|u_{n}\|_{L^{\infty}([0,T];\mathcal{H})} \leq M.
    \end{equation}
    For instance, one may take
    \begin{equation*}
        \label{eq:M}
        M=
        \big(\|u_{0}\| + 2\|\dot{a}\|_{L^{1}([0,T])} + \|c\|_{L^{1}([0,T])}+T\big)\exp\big[\|c\|_{L^{1}([0,T])}+T\big].
    \end{equation*}

    \noindent $\circ$ \textit{Step 3 - Uniform bound on the total variation:} By \eqref{eq:derivative-cn}, the uniform bound \eqref{eq:uniform-u-bound}, and the estimate of the jumps, one has
    \begin{align*}
        \operatorname{Var}(u_{n};[0,T])
        &\leq \int_{0}^{T}\|\dot{u}_{n}(t)\|dt + 2\sum_{i=1}^{N_{n}}\alpha_{i}^{n}\\
        &\leq (1+M)\|c_{n}\|_{L^{1}([0,T])} + 2\|\dot{a}\|_{L^{1}([0,T])}.
    \end{align*}
    Consequently,
    \begin{equation}
        \label{eq:uniform-variation}
        \sup_{n\in\mathbb{N}}\operatorname{Var}(u_{n};[0,T])<+\infty.
    \end{equation} 
    More precisely, let $0\leq s<t\leq T$. There exist $i,j\in\{0,\ldots,N_{n}-1\}$ such that $s\in[t_{j}^{n},t_{j+1}^{n})$ and $t\in[t_{i}^{n},t_{i+1}^{n})$, where it may be assumed that $j<i$. By decomposing the increment of $u_{n}$ into the continuous evolution on the subintervals and the jumps at the partition points, it follows that
    \begin{align*}
        \label{eq:local-variation-1}
        \|u_{n}(t)-u_{n}(s)\|
        &\leq (1+M)\int_{s}^{t}c_{n}(\tau)d\tau
        +2\sum_{k=j+1}^{i}\alpha_{k}^{n}\nonumber\\
        &\leq (1+M)\int_{s}^{t}c_{n}(\tau)d\tau
        +2\int_{s}^{t}|\dot{a}(\tau)|d\tau
        +2\int_{t_{j}^{n}}^{t_{j+1}^{n}}|\dot{a}(\tau)|d\tau.
    \end{align*}
    Moreover, by the properties of the sequence $(c_{n})_{n}$ provided by Lemma \ref{lemma:Partition},
    \begin{align*}
        \int_{s}^{t}c_{n}(\tau)d\tau
        \leq
        \int_{s}^{t}c(\tau)d\tau + 2\int_{t_{j}^{n}}^{t_{j+1}^{n}}c(\tau)d\tau + \frac{2T}{n}.
    \end{align*}
    Hence,
    \begin{equation}
        \label{eq:local-variation}
        \|u_{n}(t)-u_{n}(s)\|
        \leq 2\int_{s}^{t}|\dot{a}(\tau)|d\tau + (1+M)\int_{s}^{t}c(\tau)d\tau + \varepsilon_{n},
    \end{equation}
    where
    \begin{align*}
        \varepsilon_{n}
        :=
        2\max_{0\leq j\leq N_{n}-1}
        \int_{t_{j}^{n}}^{t_{j+1}^{n}}|\dot{a}(\tau)|d\tau + (1+M)\left( 2\max_{0\leq j\leq N_{n}-1} \int_{t_{j}^{n}}^{t_{j+1}^{n}} c(\tau) d\tau + \frac{2T}{n} \right).
    \end{align*}
    Since $\dot{a}(\cdot),c(\cdot)\in L^{1}([0,T];\mathbb{R})$ and the mesh size of the partitions converges to zero, the absolute continuity of the Lebesgue integral yields $\varepsilon_{n}\to0$ as $n\to\infty$.

    \noindent$\circ$ \textit{Step 4 - Compactness and passage to the limit:} The estimates \eqref{eq:uniform-u-bound} and \eqref{eq:uniform-variation} give a uniform bound of $(u_{n})_{n}$ in $\operatorname{BV}([0,T];\mathcal{H})$. Therefore, by the standard compactness theorem for functions of bounded variation with values in a Hilbert space (see, for instance, Theorem 0.2.1 in \cite{MR1231975}), there exist a subsequence, still denoted by $(u_{n})_{n}$, and a function $u\colon[0,T]\to\mathcal{H}$ such that
    \begin{equation}
        \label{eq:pointwise-convergence}
        u_{n}(t)\rightharpoonup u(t)
        \qquad\textrm{weakly in }\mathcal{H},\quad\forall t\in[0,T].
    \end{equation}
    Moreover, by the weak lower semicontinuity of the norm and \eqref{eq:local-variation}, for every $0\leq s\leq t\leq T$, one has
    \begin{align*}
        \|u(t)-u(s)\|
        &\leq \liminf_{n\to+\infty}\|u_{n}(t)-u_{n}(s)\|\\
        &\leq
        2\int_{s}^{t}|\dot{a}(\tau)|d\tau
        +(1+M)\int_{s}^{t}c(\tau)d\tau.
    \end{align*}
    Since $c(\cdot),|\dot{a}(\cdot)|\in L^{1}([0,T];\mathbb{R})$, it follows that
    \begin{align*}
        u(\cdot)\in W^{1,1}([0,T];\mathcal{H}) =  W^{1,1}([0,T];\mathcal{H}).
    \end{align*}
    In particular,
    \begin{align*}
        \|\dot{u}(t)\|
        \leq 2|\dot{a}(t)|+(1+M)c(t) \qquad\textrm{for a.e. }t\in[0,T].
    \end{align*}
    Since $u_{n}(0)=u_{0}$ for every $n\in\mathbb{N}$, \eqref{eq:pointwise-convergence} also yields $u(0)=u_{0}$.

    \noindent$\circ$ \textit{Step 5 - Uniqueness:} Let $u$ and $v$ be two solutions of \eqref{eq:evolution-equation} with the same initial condition
    \begin{align*}
        u(0)=v(0)=u_{0}.
    \end{align*}
    Define
    \begin{align*}
        p(t):=-\dot{u}(t)\in A(t)u(t)
        \quad\textrm{and}\quad
        q(t):=-\dot{v}(t)\in A(t)v(t).
    \end{align*}
    Since $A(t)$ is monotone, one has
    \begin{equation*}
        \label{eq:monotonicity}
        \left\langle p(t)-q(t), u(t)-v(t) \right\rangle \geq0.
    \end{equation*}
    Since
    \begin{align*}
        p(t)-q(t)=-(\dot{u}(t)-\dot{v}(t)),
    \end{align*}
    it follows that
    \begin{equation*}
        -\left\langle\dot{u}(t)-\dot{v}(t),u(t)-v(t) \right\rangle \geq0.
    \end{equation*}
    Therefore,
    \begin{equation*}
        \frac{d}{dt}\|u(t)-v(t)\|^{2}
        = 2\left\langle \dot{u}(t)-\dot{v}(t), u(t)-v(t) \right\rangle \leq0
    \end{equation*}
    for a.e. $t\in[0,T]$. Hence,
    \begin{align*}
        \|u(t)-v(t)\|^{2}
        \leq \|u(0)-v(0)\|^{2} = 0 \qquad\forall t\in[0,T].
    \end{align*}
    Consequently, $u(t)=v(t)$ for all $t\in[0,T]$. Thus, the solution is unique.

    \noindent$\circ$ \textit{Step 6 - The variation estimate:} By the weak convergence \eqref{eq:pointwise-convergence}, the weak lower semicontinuity of the norm, and \eqref{eq:local-variation}, one has
    \begin{align*}
        \|u(t)-u(s)\|
        \leq \liminf_{n\to+\infty}\|u_{n}(t)-u_{n}(s)\|
        \leq (1+M)\int_{s}^{t}c(\tau)d\tau
        +2\int_{s}^{t}|\dot{a}(\tau)|d\tau.
    \end{align*}
    Let $K:=\max\{1+M,2\}$. Then
    \begin{equation*}
        \|u(t)-u(s)\|
        \leq K\left( \int_{s}^{t} c(\tau)d\tau + \int_{s}^{t}|\dot{a}(\tau)|d\tau \right),
    \end{equation*}
    and hence
    \begin{equation*}
        \|u(t)-u(s)\|
        \leq
        K\int_{s}^{t}\left(c(\tau)+|\dot{a}(\tau)|\right)d\tau
    \end{equation*}
    for every $0\leq s\leq t\leq T$. This completes the proof. 
\end{proof}

{\bf Acknowledgements:} {\it This work was partially supported by the EIPHI Graduate School (contract ANR-17-EURE-0002), ANID Chile under FONDECYT Regular Grant No. 1261728, and ANID-Subdirección de Capital Humano through the Doctorado Nacional 2025 Scholarship No. 21250758.}

\bibliographystyle{plain}
\bibliography{references.bib}
\end{document}